\documentclass[12pt]{article}
\usepackage{latexsym,amsmath,amsopn,amssymb,amsthm,amsfonts}

\newtheorem{theorem}{Theorem}[section]
\newtheorem{corollary}[theorem]{Corollary}

\newtheorem{definition}[theorem]{Definition}
\newtheorem{lemma}[theorem]{Lemma}

\newcommand\g{{\mathfrak g}}

\newcommand{\R}{\mathbb{R}}

\begin{document}

{\bf \large
\centerline{N.~K.~Smolentsev}

\vspace{3mm}
\centerline{Canonical almost pseudo-K\"{a}hler structures }
\centerline{on six-dimensional nilpotent Lie groups\footnote{The work was partially supported by RFBR 12-01-00873-a and by Russian President Grant supporting scientific schools SS-544.2012.1
}}}

\begin{abstract}
It is known that there are 34 classes of isomorphic connected simply connected six-dimensional nilpotent Lie groups. Of these, only 26 classes suppose left-invariant symplectic structures \cite{Goze-Khakim-Med}. In \cite{CFU2} it is shown that 14 classes of symplectic six-dimensional nilpotent Lie groups suppose compatible complex structures and, therefore, define pseudo-K\"{a}hler metrics. In this paper we show that on the remaining 12 classes of six-dimensional nilpotent symplectic Lie groups there are left-invariant almost pseudo-K\"{a}hler metrics, and we study their geometrical properties.
\end{abstract}

\section{Introduction} \label{Intr}
A manifold $M$ is said to be K\"{a}hler if it has three compatible structures: the Riemannian metric $g$, the complex structure $J$ and the symplectic form $\omega$, which possess the properties:

$$
g(JX,JY)=g(X,Y),\qquad g(X,Y)=\omega(X,JY),
$$
for any vectors fields $X,Y$ on $M$. If $J$ is a (not integrable) almost complex structure then the manifold $M$ is said to be \emph{almost pseudo-K\"{a}hler}.

The problem of determining the Hermitian or symplectic manifolds that are not K\"{a}hler has a long history and is still open. All examples that have been discovered of symplectic manifolds that do not assume K\"{a}hler metrics are nilmanifolds, i.e. compact factors of nilpotent Lie groups on discrete subgroups.
In \cite{BG} it is shown that nilmanifolds (except for a torus) do not all suppose left-invariant (positive definite) K\"{a}hler metrics. However, on some of them there can be pseudo-K\"{a}hler, i.e. symplectic, structures, for which the associated metric $g(X,Y)=\omega(X,JY)$ is pseudo-Riemannian.

Thus, the left-invariant pseudo-K\"{a}hler (or indefinite K\"{a}hler) structure $(J, \omega,g)$ on a Lie group $G$ with Lie algebra $\mathfrak{g}$  consists of a nondegenerate closed left-invariant 2-form $\omega$ and a left-invariant complex structure $J$ on $G$ which are compatible, i.e. $\omega(JX, JY) = \omega(X, Y)$, for all $X, Y \in \g$. And, besides, the associated metric $g(X,Y)=\omega(X,JY)$ is pseudo-Riemannian.
	
The complete list of the 14 classes (including the K\"{a}hler structure of a torus) of such left-invariant pseudo-K\"{a}hler structures on six-dimensional nilpotent Lie groups is given in \cite{CFU2}.
In \cite{Sm-11}, such pseudo-K\"{a}hler metrics are investigated in more detail. Explicit expressions of complex structures are given, and the curvature properties are investigated. It appears that there are multiparametrical sets of such complex structures. However, all of them have a series of general properties, namely that the associated pseudo-K\"{a}hler metric is Ricci-flat, the scalar square $g(R,R)$ of the curvature tensor $R$ is equal to zero, and the curvature tensor has some non-zero components depending only on two or, at the most, three, parameters.

The classification of left-invariant symplectic structures on six-dimensional nilpotent Lie groups is established in the paper by Goze, Khakimdjanov and Medina \cite{Goze-Khakim-Med}. According to this classification, there are 12 classes of six-dimensional nilpotent Lie groups which are symplectic but do not suppose compatible complex structures \cite{CFU2}, \cite{Sal-1}. These are the following symplectic Lie groups (where the indexing corresponds to that in \cite{Goze-Khakim-Med}):

$G_1$:
$
\begin{array}[t]{l}
\left[ e_{1},e_{2}\right] =e_{3},\quad \left[ e_{1},e_{3}\right]
=e_{4},\quad \left[ e_{1},e_{4}\right] =e_{5}, \\
\left[ e_{1},e_{5}\right] =e_{6},\quad \left[ e_{2},e_{3}\right]
=e_{5},\quad \left[ e_{2},e_{4}\right] =e_{6},\medskip  \\
\omega =e^{1}\wedge e^{6}+(1-\lambda )e^{2}\wedge e^{5}+\lambda e^{3}\wedge e^{4},\qquad \lambda \in \Bbb{R}%
\setminus \{0,1\}\, .
\end{array}
$

$G_2$:
 $\begin{array}[t]{l}
\left[ e_{1},e_{2}\right] =e_{3},\quad \left[ e_{1},e_{3}\right]
=e_{4},\quad \left[ e_{1},e_{4}\right] =e_{5}, \quad
\left[ e_{1},e_{5}\right] =e_{6},\quad \left[ e_{2},e_{3}\right]
=e_{6},\medskip \qquad  \\
\omega (\lambda )=\lambda (e^{1}\wedge e^{6}+e^{2}\wedge
e^{4}+e^{3}\wedge e^{4}-e^{2}\wedge e^{5}),\qquad \lambda \ne 0\, .
\end{array}
$

$G_3$:
$\begin{array}[t]{l}
\left[ e_{1},e_{2}\right] =e_{3},\quad \left[ e_{1},e_{3}\right]
=e_{4},\quad \left[ e_{1},e_{4}\right] =e_{5}, \quad  \left[ e_{1},e_{5}\right] =e_{6},
\medskip \qquad  \\
\omega =e^{1}\wedge e^{6}-e^{2}\wedge e^{5}+e^{3}\wedge e^{4}\, .
\end{array}
$

$G_4$:
$\begin{array}[t]{l}
\left[ e_{1},e_{2}\right] =e_{3},\quad \left[ e_{1},e_{3}\right]
=e_{4},\quad \left[ e_{1},e_{4}\right] =e_{6}, \quad
\left[ e_{2},e_{3}\right] =e_{5},\quad \left[ e_{2},e_{5}\right]
=e_{6},\medskip  \\
\omega (\lambda _{1},\lambda _{2})=\lambda _{1}e^{1}\wedge e^{4}+\lambda _{2}(e^{1}\wedge e^{5}+e^{1}\wedge e^{6}+e^{2}\wedge e^{4}+e^{3}\wedge e^{5}),\\
\quad  \lambda _{1},\, \lambda _{2}\in \Bbb{R},\, \lambda _{2}\ne 0 \, .
\end{array}
$

$G_5$:
$\begin{array}[t]{l}
\left[ e_{1},e_{2}\right] =e_{3},\ \left[ e_{1},e_{3}\right]
=e_{4},\ \left[ e_{1},e_{4}\right] =-e_{6}, \
\left[ e_{2},e_{3}\right] =e_{5},\ \left[ e_{2},e_{5}\right]
=e_{6},\medskip  \\
\begin{array}[t]{l}
\omega _{1}(\lambda _{1},\lambda _{2})=\lambda _{1}e^{1}\wedge e^{4}+\lambda _{2}(e^{1}\wedge e^{5}+e^{1}\wedge e^{6}+e^{2}\wedge e^{4}+e^{3}\wedge e^{5}),\ \\
\qquad\qquad\qquad \qquad\qquad\qquad \qquad\qquad\qquad  \lambda _{1},\, \lambda _{2}\in \Bbb{R},\, \lambda _{2}\ne 0 \, ,
\end{array}
\qquad \\
\begin{array}[t]{l}
\omega _{2}(\lambda )=\lambda (e^{1}\wedge e^{6} -2e^{1}\wedge e^{5} -2e^{2}\wedge e^{4} +e^{2}\wedge e^{6}+e^{3}\wedge e^{4} +e^{3}\wedge e^{5}),\ \\
\qquad\qquad\qquad \qquad\qquad\qquad\qquad\qquad\qquad \qquad\qquad\qquad  \lambda \ne 0\, ,
\end{array}
\quad  \\
\begin{array}[t]{l}
\omega _{3}(\lambda )=\lambda (e^{1}\wedge e^{4} -e^{1}\wedge e^{5}+e^{1}\wedge e^{6}-e^{2}\wedge e^{4}+e^{2}\wedge e^{5}+e^{2}\wedge e^{6}+ \\
\qquad\qquad\qquad \qquad\qquad\qquad \qquad\qquad\qquad +e^{3}\wedge e^{4}+e^{3}\wedge e^{5}),\ \lambda \ne 0\, ,
\end{array}
\quad  \\
\begin{array}[t]{l}
\omega _{4}(\lambda )=\lambda (2e^{1}\wedge e^{4}+e^{1}\wedge e^{6}+ 2 e^{2}\wedge e^{5}+e^{2}\wedge e^{6}+e^{3}\wedge e^{4} +e^{3}\wedge e^{5}),\ \\
\qquad\qquad\qquad \qquad\qquad\qquad \qquad\qquad\qquad \qquad\qquad\qquad  \lambda \ne 0\, .
\end{array}
\end{array}
$

$G_6$:
$\begin{array}[t]{l}
\left[ e_{1},e_{2}\right] =e_{3},\quad \left[ e_{1},e_{3}\right]
=e_{4},\quad \left[ e_{1},e_{4}\right] =e_{5}, \quad
\left[ e_{2},e_{3}\right] =e_{6},
\medskip \quad  \\
\omega _{1}=e^{1}\wedge e^{6}+e^{2}\wedge e^{4}+e^{2}\wedge e^{5}-e^{3}\wedge e^{4}, \\
\omega _{2}=-e^{1}\wedge e^{6}-e^{2}\wedge e^{4}-e^{2}\wedge e^{5}+e^{3}\wedge e^{4}
\end{array}
$

$G_7$:
$\begin{array}[t]{l}
\left[ e_{1},e_{2}\right] =e_{4},\quad \left[ e_{1},e_{4}\right]
=e_{5},\quad \left[ e_{1},e_{5}\right] =e_{6}, \quad
\left[ e_{2},e_{3}\right] =e_{6},\quad \left[ e_{2},e_{4}\right]
=e_{6},\medskip  \\
\omega _{1}(\lambda )=\lambda (e^{1}\wedge e^{3}+e^{2}\wedge e^{6}-e^{4}\wedge e^{5}),\quad \lambda \ne 0 \, , \\
\omega _{2}(\lambda )=\lambda (e^{1}\wedge e^{6}+e^{2}\wedge e^{5}-e^{3}\wedge e^{4}),\quad \lambda \ne 0 \, .
\end{array}
$

$G_8$:
$\begin{array}[t]{l}
\left[ e_{1},e_{3}\right] =e_{4},\quad \left[ e_{1},e_{4}\right]
=e_{5},\quad \left[ e_{1},e_{5}\right] =e_{6}, \quad
\left[ e_{2},e_{3}\right] =e_{5},\quad \left[ e_{2},e_{4}\right]
=e_{6},\medskip  \\
\omega =e^{1}\wedge e^{6}+e^{2}\wedge e^{5}-e^{3}\wedge e^{4}
\end{array}
\qquad $

$G_9$:
$\begin{array}[t]{l}
\left[ e_{1},e_{2}\right] =e_{4},\quad \left[ e_{1},e_{4}\right]
=e_{5},\quad \left[ e_{1},e_{5}\right] =e_{6}, \quad
\left[ e_{2},e_{3}\right] =e_{6},
\medskip \qquad  \\
\omega (\lambda )=\lambda (e^{1}\wedge e^{3}+e^{2}\wedge
e^{6} -e^{4}\wedge e^{5}),\quad \lambda \ne 0\, .
\end{array}
$

$G_{19}$:
$\begin{array}[t]{l}
\left[ e_{1},e_{2}\right] =e_{4},\quad \left[ e_{1},e_{4}\right]
=e_{5},\quad \left[ e_{1},e_{5}\right] =e_{6},\medskip  \\
\omega =e^{1}\wedge e^{3}+e^{2}\wedge e^{6} -e^{4}\wedge e^{5}
\end{array}
$

$G_{20}$:
$\begin{array}[t]{l}
\left[ e_{1},e_{2}\right] =e_{3},\quad \left[ e_{1},e_{3}\right]
=e_{4},\quad
\left[ e_{1},e_{4}\right] =e_{5},\quad \left[ e_{2},e_{3}\right]
=e_{5},\medskip  \\
\omega _{1}=e^{1}\wedge e^{6}+e^{2}\wedge e^{5}-e^{3}\wedge e^{4}\, .
\end{array}
$

$G_{22}$:
$ \begin{array}[t]{l} \left[ e_{1},e_{2}\right] =e_{5},\quad \left[ e_{1},e_{5}\right] =e_{6}, \qquad  \omega =e^{1}\wedge e^{6}+e^{2}\wedge e^{5}+e^{3}\wedge e^{4}.
\end{array}$

\

In this paper we will consider these six-dimensional nilpotent Lie groups which suppose symplectic structures but do not suppose compatible complex structures.  These groups therefore do not suppose even pseudo-K\"{a}hler metrics. However, they can have left-invariant almost (pseudo) K\"{a}hler structures, i.e. structures in the form $(g,J,\omega)$, where $\omega$ is a left-invariant symplectic form, $J$ is a left-invariant almost complex structure compatible with $\omega$ and $g(X,Y) = \omega (X,JY)$, the associated (pseudo) Riemannian metric.
	
Generally speaking, for a given symplectic structure $\omega$ there is a multiparametrical set of almost complex structures. For such a general compatible almost complex structure $J$ the associated metric $g$ has a very complicated appearance. However, not all elements $\psi_{ij}$ of the matrixes of the operator $J=(\psi_{ij})$ influence the curvature of the metric $g$. For example, for the abelian group $\R^6$ with symplectic structure $\omega = e^1\wedge e^2 +e^3\wedge e^4 +e^5\wedge e^6$, there is a 12-parametrical set of almost complex structures $J=(\psi_{ij})$ that are compatible with $\omega$. However the curvature of the associated metric $g(X,Y)=\omega(X,JY)$ will be zero for any complex structure $J$. Therefore, from the geometrical point of view, there is no sense in considering the general compatible complex structure $J$. It is much more natural to choose the elementary one: $J(e_1)=e_2$,\ $J(e_3)=e_4$,\ $J(e_5)=e_6$. We will adhere to this point of view for all Lie algebras $\g$.

After finding a multiparametrical set of almost complex structures $J=(\psi_{ij})$ that is compatible with the symplectic structure $\omega$ on $\g$,
we will consider as zeros all free parameters $\psi_{ij}$ on which the Ricci tensor of the associated metric does not depend to be zero. We will call such structures \emph{canonical}.

As is known \cite{CFU2}, left-invariant pseudo-K\"{a}hler structures on six-dimensional nilpotent Lie groups have nilpotent complex structures $J$, i.e. structures which suppose a series of nontrivial $J$-invariant ideals of the Lie algebra $\g$. Therefore we will also require that the almost complex structure $J$ shows invariancy of some ideals.
Under the specified conditions, an operator $J$ and the associated metric $g$ are in an explicit form, depend on several parameters and have a simple enough appearance.

In this paper we will find the canonical almost pseudo-K\"{a}hler structures $(g,J,\omega)$ for all the 12 classes of six-dimensional nilpotent symplectic Lie algebras specified above and we will describe their geometrical properties.

All computations are fulfilled in system of computer mathematics Maple.
Formulas for computations are specified in the end of this paper.

\section{Almost pseudo-K\"{a}hler structures} \label{AlmPsKahl}
Let $G$ be a real Lie group of dimension $n=2m$ with Lie algebra $\g$.
A left-invariant almost pseudo-K\"{a}hler (or indefinite almost-K\"{a}hler) structure $(J,\omega)$ on a Lie group $G$ with Lie algebra $\g$ consists of a left-invariant nondegenerate closed 2-form $\omega$ and a left-invariant almost complex structure $J$ on $G$ which are compatible, i.e. $\omega(JX,JY)=\omega(X,Y)$, for all $X,Y\in \g$. Given a pseudo-K\"{a}hler structure $(J,\omega)$ on $\g$, there exists an associated left-invariant pseudo-Riemannian metric $g$ on $\g$ defined by
\[
g(X,Y) =\omega(X,JY), \mbox{ for } X,Y\in \g.
\]

In this paper we will study a left-invariant almost pseudo-K\"{a}hler structure $(J,\omega)$ on a Lie group $G$. As the symplectic structure $\omega$ and the almost complex structure $J$ are left-invariant on $G$, they are defined by their values on the Lie algebra $\g$. Therefore from now on we will deal only with the Lie algebra $\g$ and we will define $\omega$ and $J$ as the symplectic and, respectively, the almost complex structures on the Lie algebra $\g$.
A symplectic form on a Lie algebra $\g$ is a nondegenerate 2-form on $\g$, satisfying to an condition: $\omega([X,Y],Z)-\omega([X,Z],Y) +\omega([Y,Z](X)=0$, for all $X,Y,Z\in\g $.
An almost complex structure $J$ on a Lie algebra $\g$ is an endomorphism $J:\g \to \g$ that satisfies $J^2=-I$, where $I$ is the identity map.

The descending central series $\{C^k\g\}$ of $\g$ is defined inductively by
\[
C^0\g=\g,\quad C^{k+1}\g=[\g,\ C^k\g],\ k\ge 0.
\]
The Lie algebra $\g$ is said to be \emph{nilpotent} if $C^k\g =0$  for some $k$. In this case the minimum length of the descending central series is said to be a \emph{class} (or a \emph{step}) of nilpotencies, i.e. the Lie algebra class is equal to $s$, if both $C^s\g =0$  and $C^{s-1}\g \ne 0$. The Lie group is said to be nilpotent if its Lie algebra is nilpotent.

If $\g$ is $s$-step nilpotent, then the ascending central series $\mathfrak{g}_0=\{0\}\subset \mathfrak{g}_1 \subset \mathfrak{g}_2\subset\dots \subset \mathfrak{g}_{s-1}\subset \mathfrak{g}_s =\mathfrak{g}$ of $\g$ is defined inductively by
\[
\mathfrak{g}_k=\{X\in \mathfrak{g} |\ [X,\mathfrak{g}]\subseteq  \mathfrak{g}_{k-1}\}, \ k\geq 1.
\]
It is well-known that the sequence $\{\g_k\}$ increases strictly until $\g_s$ and, in particular, that the ideal $\g_1$ is the center $\mathcal{Z}$ of Lie algebra $\g$. We will call the sequence of dimensions $(d_1,d_2,\dots, d_s)$ of ideals $\mathfrak{g}_i$  as a \emph{type} of the Lie algebra $\mathfrak{g}$.

Since the spaces $\{\g_k\}$ are not, in general, $J$-invariant, the sequence is not suitable for working with $J$. We introduce a new sequence $\{\mathfrak{a}_l(J)\}$ having the property of $J$-invariance \cite{CFU2}. The ascending series $\{\mathfrak{a}_l(J)\}$ of the Lie algebra  $\g$, compatible with the left-invariant complex structure $J$ on $G$, is defined inductively as follows: $\frak{a}_0(J)=0$,
$$
\frak{a}_l(J) = \{X\in \mathfrak{g}\ |\ [X,\mathfrak{g}]\subseteq  \frak{a}_{l-1}(J) \mbox{ and } [JX,\mathfrak{g}]\subseteq  \frak{a}_{l-1}(J) \}\ l\ge 1.
$$
It is clear that each $\{\mathfrak{a}_l(J)\}$ is a $J$-invariant ideal in $\g$ and that $\mathfrak{a}_k(J)\subset \mathfrak{g}_k$ for $k>0$. In particular, the $J$-invariant ideal $\mathfrak{a}_1(J)$ has dimension not less than two and lies in the Lie algebra center $\mathcal{Z}$. It must be noted that the terms $\mathfrak{a}_l(J)$ depend on the complex structure $J$ considered on $G$. Moreover, this ascending series, in spite of $\g$ being nilpotent, can stop without reaching the Lie algebra $\g$, that is, it may happen that $\mathfrak{a}_l(J)\ne \g$ for all $l$. The following definition is motivated by this fact.
	
\begin{definition} \label{Def1}
The left-invariant complex structure $J$ on G is called nilpotent if there is a number $p$ such that $\mathfrak{a}_p(J)=\g$.
\end{definition}

In \cite{CFU2} it is shown that left-invariant pseudo-K\"{a}hler structures on six-dimensional nilpotent Lie groups have nilpotent complex structures. We study the Lie algebras which do not suppose complex structures. Besides, the Lie algebra center can be one-dimensional. For these the sequence of ideals $\frak{a}_k(J)$ is not defined. However many Lie algebras in this paper have two-dimensional and four-dimensional ideals. It is natural to consider almost complex structures $J$ which leave such ideals invariant.
	
\begin{definition} \label{Def2}
We will call a left-invariant almost complex structure $J$ on a nilpotent Lie group $G$ almost nilpotent if it has increasing sequence of $J$-invariant ideals $\mathfrak{b}_1(J) \subset \mathfrak{b}_2(J)\subset\dots \subset \mathfrak{b}_{m-1}(J)\subset \mathfrak{b}_m =\mathfrak{g}$, of dimensions $2,4,\dots,2m$.
\end{definition}

Let us establish some properties of a symplectic nilpotent Lie algebra $(\g,\omega)$ with a compatible almost complex structure $J$.

\begin{lemma} \label{C1ortZ}
Let $\mathcal{Z}$ be the Lie algebra center, then for any symplectic form $\omega$ on $\g$, the equality   $\omega(C^1\mathfrak{g},\mathcal{Z}) =0$ is fulfilled.
\end{lemma}

\begin{proof}
This follows at once from the formula $d\omega(X,Y,Z)=\omega([X,Y],Z) -\omega([X,Z],Y) +\omega([Y,Z],X)=0$, \ $\forall X,Y\in\mathfrak{g}$, $\forall Z\in\mathcal{Z}$.
\end{proof}

\begin{corollary}
If $J$ is a compatible almost complex structure, then $g(J(C^1\mathfrak{g}),\mathcal{Z}) =g(C^1\mathfrak{g},J(\mathcal{Z})) =0$ for the associated metric $g(X,Y)=\omega(X,JY)$.
\end{corollary}

\begin{corollary}
If the ideal $C^k\g$,  $k>0$, is $J$-invariant, then any vector $X\in\mathcal{Z}\cap C^k\g$ is isotropic for the associated metric $g(X,Y)=\omega(X,JY)$. In particular, the associated metric is pseudo-Riemannian.
\end{corollary}

\begin{corollary}
If an ideal $\mathfrak {b}$ is $J$-invariant and $\mathfrak{b}\subset\mathcal{Z}$, then $\omega(C^1\mathfrak{g}\oplus J(C^1\mathfrak{g}),\mathfrak{b}) =0$ for any compatible almost complex structure $J$.
\end{corollary}

\begin{corollary} \label{1.6}
If an ideal $\mathfrak {b}$ is $J$-invariant and $\mathfrak{b}\subset\mathcal{Z}$, then for any almost (pseudo) K\"{a}hler structure $(\mathfrak{g}, \omega, g, J)$ the ideal $\mathfrak{b}$ is orthogonal to a subspace $C^1\mathfrak{g}\oplus J(C^1\mathfrak{g})$:
$$
g(C^1\mathfrak{g}\oplus J(C^1\mathfrak{g}),\mathfrak{b}) =0.
$$
\end{corollary}

From the formula
\begin{equation}\label{nabla}
2g(\nabla_X Y, Z)= g([X,Y],Z) + g([Z,X],Y)  +g(X,[Z,Y]),\ \forall X,Y,Z\in\g
\end{equation}
for a covariant derivative $\nabla$ of Levi-Civita connection for left-invariant vector fields on a Lie group, the following properties are at once implied:
\begin{itemize}
\item if the vectors $X$ and $Y$ lie in the center $\mathcal{Z}$ of the Lie algebra $\g$, then $\nabla_X Y=0$ for any left-invariant (pseudo) Riemannian metric $g$ on the Lie algebra;
\item if the vector $X$ lies in the center $\mathcal{Z}$ of the Lie algebra $\g$, then $\nabla_X Y = \nabla_Y X$.
\end{itemize}

\begin{lemma} \label{nablaXY}
Let $J$ be a compatible almost complex structure on a symplectic nilpotent Lie algebra $(\g, \omega)$  and $g=\omega \cdot J$ be the associated (pseudo)Riemannian metric. Then for a Levi-Civita covariant derivative on a Lie group, the following properties are fulfilled:
\begin{itemize}
\item if the vector $X$ lies in the center $\mathcal{Z}$ of the Lie algebra $\g$ and the ideal $C^1\mathfrak{g}$ is $J$-invariant, then  $\nabla_X Y = \nabla_Y X =0$ for any $Y\in \mathfrak{g}$;
\item if the vector $X$ lies in a $J$-invariant ideal $\mathfrak{b}\subset \mathcal{Z}$, then $\nabla_X Y=\nabla_Y X=0$,\ $\forall Y\in \mathfrak{g}$.
\end{itemize}
\end{lemma}

\begin{proof}
Let, for example, $X\in \mathcal{Z}$ and $Z,Y\in \mathfrak{g}$. Then the formula (\ref{nabla}) and Lemma \ref{C1ortZ} imply that $2g(\nabla_X Y,Z) = g([Z,Y],X)=-\omega(J[Z,Y],X) =0$.
\end{proof}

\begin{lemma}\label{ParRiem}
Let $J$ be a compatible almost complex structure on the symplectic nilpotent Lie algebra $(\g, \omega)$, and let $g=\omega \cdot J$ be the associated (pseudo) Riemannian metric. Let us choose a basis $\{e_1, e_2,\dots e_{p-1}, e_{p},\dots, e_{2m}\}$ of the Lie algebra $\g$, such that $\{e_{p},\dots ,e_{2m}\}$ is a basis of the center $\mathcal{Z}$. Let $J = (\psi_{ij})$ be the matrix of $J$ in this basis. Then for all $X,Y\in \g$, the covariant derivative $\nabla_X Y$ does not depend on the free parameters $\psi_{ij}$, where $i=p,\dots 2m$, and for all $j$.
\end{lemma}

\begin{proof}
Let $X\in \g$ and $JX = J_1X +J_cX$, where $J_1X \in \mathbb{R}\{e_1, e_2,\dots, e_{p-1}\}$ and $J_cX\in \mathbb{R}\{e_{p},\dots, e_{2m}\} = \mathcal{Z}$. From the formula(\ref{nabla}) for a covariant derivative it follows that for all $Z\in \g$:
$$
2\omega(\nabla_X Y, JZ)= \omega([X,Y],JZ) + \omega([Z,X],JY)  +\omega([Z,Y],JX),
$$
$$
2\omega(\nabla_X Y, Z)= \omega([X,Y],Z) - \omega([JZ,X],JY)  -\omega([JZ,Y],JX)=
$$
$$
= \omega([X,Y],Z) + \omega([J_1Z+J_cZ,X],J_1Y+J_cY)  +\omega([J_1Z+J_cZ,Y],J_1X+J_cX)=
$$
$$
=\omega([X,Y],Z) + \omega([J_1Z,X],J_1Y)  +\omega([J_1,Y],J_1X).
$$
The last equality is implied from $J_cX\in \mathcal{Z}$ and from Lemma \ref{C1ortZ}. As the component $J_1$ does not depend on the parameters $\{\psi_{ij}\}$, $i=p,\dots 2m,\, j=1,\dots, 2m$, the covariant derivative $\nabla_X Y$ also does not depend on these parameters.
\end{proof}

Let us consider a curvature tensor of the associated metric
\begin{equation}\label{CurvTensor}
R(X,Y)Z=\nabla_X \nabla_Y Z -\nabla_Y \nabla_X Z -\nabla_{[X,Y]}Z.
\end{equation}

\begin{corollary} \label{ParCurvTensor}
Under the conditions of Lemma \ref{ParRiem} the curvature tensor $R(X,Y)Z$ does not depend on the free parameters $\psi_{ij}$, $i=p,\dots 2m,\, j=1,\dots, 2m$.
\end{corollary}

{\bf Remark 1.}
We note that the parameters $\psi_{ij}$ of the complex structure $J$ are restricted by two conditions: a compatibility condition and $J^2=-1$. Therefore some of the parameters $\psi_{ij}$ specified above can be expressed through others. In Lemma \ref{ParRiem} and as a consequence of \ref{ParCurvTensor} this is a question of free parameters, i.e. parameters which remain independent. The curvature does not depend on them.
 \vspace{1mm}

For many six-dimensional nilpotent Lie algebras $\mathfrak{g}$, there is an increasing sequence of ideals $\mathfrak{b}_1 \subset \mathfrak{b}_2\subset \mathfrak{b}_3 =\mathfrak{g}$ of dimensions 2, 4, and 6. Thus $\mathfrak{b}_2 = C^1\g$. Besides, if $\omega$ is set on the $\mathfrak{g}$ symplectic structure then it is possible to decompose  $\mathfrak{g}$ into the direct sum of two-dimensional subspaces:
$$
\mathfrak{g}=A\oplus B \oplus C,
$$
possessing the properties:
\begin{itemize}
\item $C=\mathfrak{b}_1$ and the form $\omega$ vanishes on $C$,
\item $B\oplus C=C^1\mathfrak{g}$  and the form $\omega$ is nondegenerate on $B$,
\item $[A,A]\subset B\oplus C$,\ $[A,B\oplus C]\subset C$ and the form $\omega$ vanishes on $A$ and $C$.
\end{itemize}

The basis $\{e_1,\dots, e_6\}$ of $\mathfrak{g}$ is chosen so that $\{e_1,e_2\}$, $\{e_3,e_4\}$ and $\{e_5,e_6\}$ are bases of subspaces $A$, $B$ and $C$ respectively. If the ideals $C^1\mathfrak{g}$ and $\mathfrak{b}_1$ are $J$-invariant, then $J$ has the following block type:
\begin{equation} \label{J6Bl}
J_0=  \left( \begin {array}{cccccc} \psi_{11}&\psi_{12}&0&0&0&0 \\
\psi_{21}&\psi_{22}&0&0&0&0\\
\psi_{31}&\psi_{32}&\psi_{33}&\psi_{34} &0&0\\
\psi_{41}&\psi_{42}&\psi_{43}&\psi_{44} &0&0\\
\psi_{51}&\psi_{52}&\psi_{53}&\psi_{54} &\psi_{55}&\psi_{56}\\
\psi_{61}&\psi_{62}&\psi_{63}&\psi_{64} &\psi_{65}&\psi_{66}
\end {array} \right).
\end{equation}

The subspace $W\subset \mathfrak{g}$ is called $\omega$-\emph{isotropic} if and only if $\omega(W,W)=0$. We will call subspaces $U,V\subset \mathfrak{g}$  $\omega$-\emph{dual} if, for any vector $X\in U$, there is a vector $Y\in V$ such that $\omega(X,Y)\ne 0$ and, on the contrary, $\forall Y\in V$, $\exists X\in U$, что $\omega(X,Y)\ne 0$.	

\begin{theorem} \label{246}
Let the six-dimensional symplectic Lie algebra $(\mathfrak{g}, \omega)$ with the compatible almost complex structure $J$ have an expansion on three two-dimensional subspaces
\begin{equation}
\mathfrak{g}=A\oplus B \oplus C,
\label{abcAK}
\end{equation}
where $B \oplus C = C^1\mathfrak{g}$ and $C$ are $J$-invariant abelian ideals. We will assume that the subspaces $A$ and $C$ are $\omega$-isotropic and $\omega$-dual, and that on the subspace $B$ the form of $\omega$ is nondegenerate and $\omega(B,C)=0$. Then for any complex structure $J$ compatible with $\omega$, and with the Levi-Civita connection $\nabla$ of the associated pseudo-Riemannian metric $g=\omega \cdot J$, the following properties hold:
\begin{itemize}
\item $g(B\oplus C, C)=0;$
\item $\nabla_X Y,\, \nabla_Y X\in B \oplus C,\quad \forall X\in \mathfrak{g}$, $\forall  Y\in B\oplus C$,
\item $\nabla_X Y,\,\nabla_Y X\in C,\quad \forall X\in \mathfrak{g},\ Y\in C$,
\item  $\nabla_X Y =\nabla_Y X\in C,\quad \forall X,Y\in B\oplus C$.
\item  $\nabla_X Y=\nabla_Y X=0,\quad \forall X\in B\oplus C,\ Y\in C$.
\end{itemize}
\end{theorem}

\begin{proof}
The first property is obvious. Let us prove the second property.
Let $X\in \mathfrak{g}$,\ $Y\in B\oplus C$. If $\nabla_X Y$ has a non-zero component from $A$ then there is a vector $JZ\in C$, such that $\omega(\nabla_X Y,JZ)\ne 0$. On the other hand,
$$
2\omega(\nabla_X Y,JZ) = 2g(\nabla_X Y,Z) = g([X,Y],Z) +g([Z,X],Y) +g([Z,Y],X) =
$$
$$
=\omega([X,Y],JZ) +\omega([Z,X],JY) +\omega(X,J[Z,Y]) = 0,
$$
as $[X,Y]\in C$, $JZ\in C$ and $\omega(C,C)=0$; $[Z,X]\in C$, $JY\in B\oplus C$ and $\omega(B\oplus C,C)=0$; $[Z,Y]=0$ from the commutability condition $B\oplus C$.

Let us consider the third property. Let $X\in \mathfrak{g}$,\ $Y\in C$. Then, from the second property we see that: $\nabla_X Y\in B\oplus C$. If $\nabla_X Y$ has a non-zero component from $B$ then there is a vector $JZ\in B$, such that $\omega(\nabla_X Y,JZ)\ne 0$. On the other hand,
$2\omega(\nabla_X Y,JZ) = 2g(\nabla_X Y,Z) = g([X,Y],Z) +g([Z,X],Y) +g([Z,Y],X) = 0$, as $[X,Y],\, Y\in C$ and $Z,\, [Z,X]\in B\oplus C$, and $g(B\oplus C,C) =0$, and from the commutability of ideal $B\oplus C$.

Now let $X,\, Y\in B\oplus C$. Then $\nabla_X Y\in B\oplus C$. If $\nabla_X Y$ has a non-zero component from $B$ then there is a vector $JZ\in B$, such that $\omega(\nabla_X Y,JZ)\ne 0$.  On the other hand,
$2\omega(\nabla_X Y,JZ) = 2g(\nabla_X Y,Z) = g([X,Y],Z) +g([Z,X],Y) +g([Z,Y],X) =0$, as  $Z\in B\oplus C$, and from the commutability of ideal $B\oplus C$.

Last statement. Let  $X\in B\oplus C$ and $Y\in C$. Then $\nabla_X Y\in C$. If $\nabla_X Y \ne 0$, then there is a vector $JZ\in A$, such that $\omega(\nabla_X Y,JZ)\ne 0$. On the other hand,  $2\omega(\nabla_X Y,JZ) = 2g(\nabla_X Y,Z) = g([X,Y],Z) +g([Z,X],Y) +g([Z,Y],X) = 0$, as ideal $B\oplus C$ is commutative, $g(B\oplus C,C) =0$ and $[Z,Y]\in C$.
\end{proof}

\begin{corollary} \label{Ric}
Under the suppositions of Theorem \ref{246} the following properties are fulfilled:
\begin{itemize}
\item if one of the three vectors $X,Y,Z$ lies in $B\oplus C$, then  $R(X,Y)Z\in B\oplus C$,
\item if two of the three vectors  $X,Y,Z$ lie in $B\oplus C$, then  $R(X,Y)Z\in C$,
\item if one of the vectors $X, Y, Z$ lies in ideal $C$, then  $R(X,Y)Z=0$.
\end{itemize}
\end{corollary}

\begin{proof}
This follows at once from the formula for a curvature tensor
\[
R(X,Y)Z=\nabla_X \nabla_Y Z -\nabla_Y \nabla_X Z -\nabla_{[X,Y]}Z
\]
and from Lemma \ref{nablaXY}.
\end{proof}

\begin{corollary} \label{Ric}
Under the suppositions of Theorem \ref{246} the Ricci tensor $Ric(X,Y)=0$ if one of the vectors $X,Y$ lies in $B\oplus C$.
\end{corollary}

\begin{proof}
Let us choose a base $e_1,\dots, e_6$ according to expansion (\ref{abcAK}), i.e. so that $A=\mathbb{R}\{e_1,e_2\}$, $B=\mathbb{R}\{e_3,e_4\}$, $C=\mathbb{R}\{e_5,e_6\}$. The Ricci tensor is defined as the convolution of the first and upper indexes of the curvature tensor $Ric_{ij}=\sum_k R_{k,i,j}^k$. It is necessary to show that $Ric_{ij}=0$ if the index $j$ has the values $3,4,5,6$.

Let us calculate $R_{k,i,j}^s$ for $k=1,2$. For example, we have $R(e_1,e_i)e_3\in B\oplus C$.
%$$
%R(e_1,e_i)e_3=\nabla_{e_{1}}\nabla_{e_{i}}e_{3}-\nabla_{e_{i}}\nabla_{e_{1}}e_{3} -\nabla_{[e_{1},e_{i}]}e_{3}\in B\oplus C.
%$$
Therefore $R_{k,i,j}^k=0$, when $k=1,2$, а $j\in \{3,4,5,6\}$.

Let us calculate $R_{k,i,j}^s$ for $k=3,4$. For example, we have $R(e_3,e_i)e_4\in C$.
%$$
%R(e_3,e_i)e_4=\nabla_{e_{3}}\nabla_{e_{i}}e_{4}-\nabla_{e_{i}}\nabla_{e_{3}}e_{4} -\nabla_{[e_{3},e_{i}]}e_{4}\in C.
%$$
Therefore $R_{k,i,j}^k=0$, when $k=3,4$, а $j\in \{3,4,5,6\}$.

Let us calculate $R_{k,i,j}^s$ for $k=5,6$. For example, we have $R(e_5,e_3)e_i =0$.
%$$
%R(e_5,e_3)e_i=\nabla_{e_{5}}\nabla_{e_{3}}e_{i}-\nabla_{e_{3}}\nabla_{e_{5}}e_{i} -\nabla_{[e_{5},e_{3}]}e_{i} =0.
%$$
Therefore $R_{k,j,i}^k=0$, when $k=5,6$, а $j\in \{3,4,5,6\}$. Therefore $Ric_{ij}=0$ if the index $j$ has values $3,4,5,6$.
\end{proof}

 \vspace{1mm}
\textbf{Remark 2.} Theorem \ref{246} and Corollary \ref{Ric} imply that the Ricci tensor is generally speaking not zero, but has a block type with one, probably non-zero, left upper block of an order 2. The scalar curvature is equal to zero
 \vspace{1mm}

\begin{corollary} \label{RicPar}
In the conditions of Theorem \ref{246} the Ricci tensor $Ric(X,Y)$ of the associated metric $g(X,Y)=\omega(X,JY)$ does not depend on the free parameters $\psi_{i1}$, $\psi_{i2}$, $i=3,4,5,6$ and $\psi_{i3}$, $\psi_{i4}$, $i=5,6$.
\end{corollary}

 \vspace{1mm}
\textbf{Remark 3.}
If on a six-dimensional nilpotent Lie algebra $\g$ an almost nilpotent almost complex structure $J$ is chosen with a sequence of $J$-invariant ideals $\mathfrak{b}_1\subset\mathfrak{b}_2\subset \mathfrak{g}$ of dimensions $(2,4,6)$, then such Lie algebra is decomposed into the direct sum of two-dimensional subspaces:
$$
\mathfrak{g}=A\oplus B\oplus \mathfrak{b}_1,
$$
where $B \oplus \mathfrak{b}_1=\mathfrak{b}_2$ and $[A,A]\subset B \oplus \mathfrak{b}_1$.  On the majority of Lie groups under consideration the symplectic form and almost complex structure satisfy the conditions of Theorem \ref{246} for this expansion.

\section{Almost pseudo-K\"{a}hler Lie groups}
In this part we will consider in detail each six-dimensional symplectic Lie group which does not suppose left-invariant compatible complex structures, but supposes compatible almost nilpotent almost complex structures and, thus, is an almost pseudo-K\"{a}hler Lie group.

\subsection{The Lie group $G_{1}$}
Let us consider a six-dimensional symplectic Lie group $G_{1}$ which has a Lie algebra $\g_1$  with non-trivial Lie brackets \cite{Goze-Khakim-Med}:

$[e_1,e_2] = e_3$, $[e_1,e_3] = e_4$,  $[e_1,e_4] = e_5$, $[e_1,e_5] = e_6$, $[e_2,e_3] = e_5$, $[e_2,e_4] = e_6$,
\vspace{1mm}\\
%$\g_{1}=(0,0,12,13,14+23,15+24)$.\\
The sequence of ideals is:
$C^1\g= \R\{e_3,e_4,e_5,e_6\}$,
$C^2\g=\R\{e_4,e_5,e_6\}$,
$C^3\g=\R\{e_5,e_6\}$,
$C^4\g=\R\{e_6\}=\mathcal{Z}$, and the last is the Lie algebra center.
This Lie algebra is filiform and consequently does not suppose an integrable complex structure \cite{Goze-E}.
There is an increasing series of ideals: $\g_1=\mathcal{Z}=\R\{e_6\}$,\, $\g_2=\R\{e_5,e_6\}$,\,
$\g_3=\R\{e_4,e_5,e_6\}$,\, $\g_4=\R\{e_3,e_4,e_5,e_6\}$,\,
$\g_5=\g$. The Lie algebra type is (1,2,3,4,6).

The symplectic structure is defined by the 2-form \cite{Goze-Khakim-Med}:
$$
\omega = e^1\wedge e^6 +(1-t)e^2\wedge e^5  + te^3\wedge e^4, \ \mbox{ где }  t\ne 0 \ \mbox{ и } t\ne 1.
$$

This Lie algebra has two even-dimensional ideals $C^1\g$ and $C^3\g$ which allow us to present a Lie algebra in the form of the direct sum $\mathfrak{g}=A\oplus B \oplus C$, satisfying the conditions of Theorem \ref{246}: $C= C^3\g$, $B\oplus C = C^1\g$ and $A=\{e_1,e_2\}$, $B=\{e_3,e_4\}$,  $C=\{e_5,e_6\}$.

If we require from an almost complex structure $J$ that it is compatible with a symplectic structure $\omega$ and that the ideals $C^1\g$ and $C^3\g$ are $J$-invariant, then we obtain the following almost nilpotent almost complex structure:
$$
J= \left[ \begin {array}{cccccc} {\it \psi_{11}}&{\it \psi_{12}}&0&0&0&0
\\ \noalign{\medskip}{\it \psi_{21}}&{\it \psi_{22}}&0&0&0&0
\\ \noalign{\medskip}{\it \psi_{31}}&{\it \psi_{32}}&{\it \psi_{33}}&{\it \psi_{34}}&0
&0\\ \noalign{\medskip}{\it \psi_{41}}&{\it \psi_{42}}&{\it \psi_{43}}&-{\it \psi_{33}
}&0&0\\ \noalign{\medskip}{\it \psi_{51}}&{\it \psi_{52}}&-{\frac {{\it \psi_{42}}
\,t}{t-1}}&{\frac {{\it \psi_{32}}\,t}{t-1}}&-{\it \psi_{22}}&{\frac {{\it
\psi_{12}}}{t-1}}\\ \noalign{\medskip}{\it \psi_{61}}&-{\it \psi_{51}}\, \left( t-
1 \right) &{\it \psi_{41}}\,t&-{\it \psi_{31}}\,t& \left( t-1 \right) {\it
\psi_{21}}&-{\it \psi_{11}}\end {array} \right].
$$
In this case the associated metric will be pseudo-Riemannian. We can note that the Riemannian associated metric also exists, but that the ideal $C^1\g$ will not be $J$-invariant. For example, if we suppose that $J(e_1)=e_6$, $J(e_2)=e_5$, $J(e_3)=e_4$ then for $t\in (0,1)$ we obtain the Riemannian associated metric $g(X,Y)=\omega(X,JY)$ and the almost K\"{a}hler structure $(g,J,\omega)$.

This Lie algebra has the one-dimensional center $\mathcal{Z}=\mathbb{R}\{e_6\}$. According to Lemma \ref{ParRiem} and Corollary \ref{ParCurvTensor} the curvature tensor does not depend on the free parameters $\psi_{6,i}$. Therefore we assume that $\psi_{61} =0$.

The condition $J^2=-1$ gives a complicated expression for the almost complex structure that depends on seven parameters (under the condition that $\psi_{61} =0$). Nevertheless, the curvature tensor $R_{ijk}^s$ is easily calculated (on Maple), but has many non-zero and complicated components. The scalar square $g(R,R)=0$ of the curvature tensor $R$ is equal to zero.

The Ricci tensor $Ric_{jk}=R_{sjk}^s$ has a block appearance with one non-zero left upper block of an order 2:
$$
Ric=-\frac {1}{2} \left[ \begin {array}{cccccc} {\frac {{{\it \psi_{11}}}^{2}{{\it \psi_{12}}}^{2}+{t}^{2}{{\it \psi_{34}}}^{2}-2\,t{{\it \psi_{34}}}^{2}+{{\it \psi_{34}
}}^{2}}{ \left( t-1 \right) ^{2}}}&{\frac {{\it \psi_{11}}\,{{\it
\psi_{12}}}^{3}}{ \left( t-1 \right) ^{2}}}&0&0&0&0\\ \noalign{\medskip}{\frac {{\it \psi_{11}}\,{{\it \psi_{12}}}^{3}}{ \left( t-1 \right) ^{2}}}&{\frac {{{\it \psi_{12}}}^{4}}{ \left( t-1 \right) ^{2}}}&0&0&0&0
\\ \noalign{\medskip}0&0&0&0&0&0\\ \noalign{\medskip}0&0&0&0&0&0
\\ \noalign{\medskip}0&0&0&0&0&0\\ \noalign{\medskip}0&0&0&0&0&0
\end {array} \right].
$$
This depends on three parameters $\psi_{11}$, $\psi_{12}\ne 0$ and $\psi_{34}\ne 0$.
The scalar curvature $S$ is equal to zero.

Under the condition that
$\psi_{34}=\pm \frac{\psi_{12}}{t-1}$ the Ricci tensor is $J$-Hermitian, $Ric(JX,JY)=Ric(X,Y)$. We suppose that $\psi_{34}=+\frac{\psi_{12}}{t-1}$. Then (irrespective of the sign) the Ricci tensor has a block appearance with one non-zero left upper block $Ric_{2\times 2}$:
$$
Ric_{2\times 2}= -\frac 12 \left[ \begin {array}{cc} {\frac { \left( 1+{{\it \psi_{11}}}^{2} \right) {{\it \psi_{12}}}^{2}}{ \left( t-1 \right) ^{2}}}& {\frac {{\it \psi_{11}}\,{{\it \psi_{12}}}^{3}}{ \left( t-1 \right) ^{2}}} \\
\noalign{\medskip}{\frac {{\it \psi_{11}}\,{{\it \psi_{12}}}^{3}}{
 \left( t-1 \right) ^{2}}}&{\frac {{{\it \psi_{12}}}^{4}}{ \left(t-1\right) ^{2}}}
\end {array} \right].
$$
For a canonical almost pseudo-K\"{a}hler structure, we set all parameters on which the Ricci tensor does not depend equal to zero. In addition, we will require that the Ricci tensor of the associated metric has the $J$-Hermitian property. Then we obtain the following expressions for a canonical almost complex structure and associated almost pseudo-K\"{a}hler metric with a Hermitian Ricci tensor:
\[
J(e_2)= \psi_{12}\, e_1 -\psi_{11}\, e_2,\qquad
J(e_4)= \frac{\psi_{12}}{t-1}\, e_3,\qquad
J(e_6)= \frac{\psi_{12}}{t-1}\, e_5 -\psi_{11}\, e_6.
\]
%$$
%J= \left[ \begin {array}{cccccc} {\it \psi_{11}}&{\it \psi_{12}}&0&0&0&0
%\\ \noalign{\medskip}-{\frac {1+{{\it \psi_{11}}}^{2}}{{\it \psi_{12}}}}&-{
%\it \psi_{11}}&0&0&0&0\\ \noalign{\medskip}0&0&0&{\frac {{\it \psi_{12}}}{t-1}
%}&0&0\\ \noalign{\medskip}0&0&-{\frac {t-1}{{\it \psi_{12}}}}&0&0&0
%\\ \noalign{\medskip}0&0&0&0&{\it \psi_{11}}&{\frac {{\it \psi_{12}}}{t-1}}
%\\ \noalign{\medskip}0&0&0&0&-{\frac { \left( t-1 \right)  \left( 1+{{
%\it \psi_{11}}}^{2} \right) }{{\it \psi_{12}}}}&-{\it \psi_{11}}\end {array}
% \right],
%$$
$$
g=-\left[ \begin {array}{cccccc} 0&0&0&0&{\frac { \left( t-1 \right)  \left( 1+{{\it \psi_{11}}}^{2} \right) }{{\it \psi_{12}}}}& {\it \psi_{11}}
\\ \noalign{\medskip}0&0&0&0& \left( t-1 \right) {\it \psi_{11}}& {\it
\psi_{12}}\\ \noalign{\medskip}0&0& {\frac { \left( t-1 \right) t}{{\it
\psi_{12}}}}&0&0&0\\ \noalign{\medskip}0&0&0& {\frac {t{\it \psi_{12}}}{t-1}}&0 &0\\
\noalign{\medskip} {\frac { \left( t-1 \right)  \left( 1+{{\it
\psi_{11}}}^{2} \right) }{{\it \psi_{12}}}}& \left( t-1 \right) {\it \psi_{11}}&0 &0&0&0\\
\noalign{\medskip}{\it \psi_{11}}& {\it \psi_{12}}&0&0&0&0
\end {array} \right].
$$

\subsection{The Lie group $G_{2}$}
%Пересчитана
Let us consider a six-dimensional symplectic Lie group $G_2$ which has a Lie algebra $\g_2$  with non-trivial Lie brackets \cite{Goze-Khakim-Med}:

$[e_1,e_2] = e_3$, $[e_1,e_3] = e_4$,  $[e_1,e_4] = e_5$, $[e_1,e_5] = e_6$, $[e_2,e_3] = e_6$.
\vspace{1mm}\\
%$\g_{2}=(0,0,12,13,14,15+23)$.\\
The sequence of ideals is:
$C^1\g=\R\{e_3,e_4,e_5,e_6\}$,
$C^2\g=\R\{e_4,e_5,e_6\}$,
$C^3\g=\R\{e_5,e_6\}$,
$C^4\g=\R\{e_6\}=\mathcal{Z}$, and the last is the Lie algebra center.
This Lie algebra is filiform and consequently does not suppose an integrable complex structure \cite{Goze-E}.
There is an increasing series of ideals: $\g_1=\mathcal{Z}=\R\{e_6\}$,\, $\g_2=\R\{e_5,e_6\}$,\,
$\g_3=\R\{e_4,e_5,e_6\}$,\, $\g_4=\R\{e_3,e_4,e_5,e_6\}$,\,
$\g_5=\g$. The Lie algebra type is (1,2,3,4,6).

The symplectic structure is defined by the 2-form \cite{Goze-Khakim-Med}:
$$
\omega = e^1\wedge e^6 - e^2\wedge e^5 + e^2\wedge e^4 + e^3\wedge e^4.
$$
We require that the almost complex structure $J$ is $\omega$-compatible and that the ideal $C^1\g$ is $J$-invariant ($\psi_{13}=0$, $\psi_{14}=0$, $\psi_{36}=\psi_{14}$, $\psi_{16}=0$, $\psi_{23}=0$, $\psi_{24}=0$,  $\psi_{35}=-\psi_{24}$). In this case the ideal $C^3\g$ is also $J$-invariant. The almost complex structure $J$ is almost nilpotent. The associated metric will be pseudo-Riemannian. In addition, we set the free parameters corresponding to the center $\mathcal{Z}$ to be equal to zero, $\psi_{61}=0$. We note that the Riemannian associated metric also exists, but the ideal $C^1\g$ will not be $J$-invariant.

The condition $J^2=-1$ gives a complicated expression for the almost complex structure that depends on seven parameters (under the condition that $\psi_{61} =0$). Nevertheless, the curvature tensor $R_{ijk}^s$ is easily calculated, but has many non-zero and complicated components. The scalar square of the curvature tensor $R$ is equal to zero, $g(R,R)=0$.

The Ricci tensor $Ric_{jk}=R_{sjk}^s$ has a block appearance with one non-zero left upper block $Ric_{2\times 2}$ of an order 2:
\[
Ric_{2\times 2} = -\frac 12 \left[ \begin {array}{cc} {{\it \psi_{11}}}^{2}{{\it \psi_{12}}}^{2}+ {{\it \psi_{34}}}^{2}&{\it \psi_{11}}\,{{\it \psi_{12}}}^{3}\\ \noalign{\medskip}{\it \psi_{11}}\,{{\it \psi_{12}}}^{3}& {{\it \psi_{12}}}^{4}
\end {array} \right]
\]
%$$
%Ric= -\frac 12 \left[ \begin {array}{cccccc} {{\it \psi_{11}}}^{2}{{\it \psi_{12}}}^{2}+ {{\it \psi_{34}}}^{2}&{\it \psi_{11}}\,{{\it \psi_{12}}}^{3}&0&0&0&0
%\\ \noalign{\medskip}{\it \psi_{11}}\,{{\it \psi_{12}}}^{3}& {{\it
%\psi_{12}}}^{4}&0&0&0&0\\ \noalign{\medskip}0&0&0&0&0&0
%\\ \noalign{\medskip}0&0&0&0&0&0\\ \noalign{\medskip}0&0&0&0&0&0
%\\ \noalign{\medskip}0&0&0&0&0&0\end {array} \right]
%$$
This depends on three parameters $\psi_{11}$, $\psi_{12}\ne 0$ and $\psi_{34}\ne 0$. The scalar $S$ curvature is equal to zero.

Under the condition $\psi_{34}^2=\psi_{12}^2$ the Ricci tensor is $J$-Hermitian, $Ric(JX,JY)=Ric(X,Y)$. Then:
\[
Ric_{2\times 2} =-\frac 12\left[ \begin {array}{cc}
{{\psi_{12}}}^{2}(1+ {{\psi_{11}}}^{2})& {\psi_{11}}\,{{\psi_{12}}}^{3}\\
\noalign{\medskip}{\psi_{11}}\,{{\psi_{12}}}^{3}& {{\psi_{12}}}^{4}
\end {array} \right].
\]
%$$
%Ric= -\frac 12\left[ \begin {array}{cccccc} {{\it \psi_{12}}}^{2}+ {{\it
%\psi_{11}}}^{2}{{\it \psi_{12}}}^{2}& {\it \psi_{11}}\,{{\it \psi_{12}}}^{3}&0&0&0 &0\\
%\noalign{\medskip}{\it \psi_{11}}\,{{\it \psi_{12}}}^{3}& {{
%\it \psi_{12}}}^{4}&0&0&0&0\\ \noalign{\medskip}0&0&0&0&0&0
%\\ \noalign{\medskip}0&0&0&0&0&0\\ \noalign{\medskip}0&0&0&0&0&0
%\\ \noalign{\medskip}0&0&0&0&0&0\end {array} \right].
%$$
Let us assume that all remaining parameters on which the Ricci tensor does not depend are equal to zero. Then we obtain the following expressions for the canonical almost complex structure $J$ and associated almost pseudo-K\"{a}hler metric $g$ with a Hermitian Ricci tensor at $\psi_{34}=\psi_{12}$:
$$
J=  \left[ \begin {array}{cccccc} {\it \psi_{11}}&{\it \psi_{12}}&0&0&0&0
\\ \noalign{\medskip}-{\frac {1+{{\it \psi_{11}}}^{2}}{{\it \psi_{12}}}}&-{
\it \psi_{11}}&0&0&0&0\\ \noalign{\medskip}0&0&0&{\it \psi_{12}}&0&0
\\ \noalign{\medskip}0&0&-{{\it \psi_{12}}}^{-1}&0&0&0
\\ \noalign{\medskip}0&0&-{{\it \psi_{12}}}^{-1}&-{\it \psi_{11}}&{\it \psi_{11}}&
{\it \psi_{12}}\\ \noalign{\medskip}0&0&0&{\frac {1+{{\it \psi_{11}}}^{2}}{{
\it \psi_{12}}}}&-{\frac {1+{{\it \psi_{11}}}^{2}}{{\it \psi_{12}}}}&-{\it \psi_{11}}
\end {array} \right],
$$
$$
g=  \left[ \begin {array}{cccccc} 0&0&0&{\frac {1+{{\it \psi_{11}}}^{2}}{{
\it \psi_{12}}}}&-{\frac {1+{{\it \psi_{11}}}^{2}}{{\it \psi_{12}}}}&-{\it \psi_{11}}
\\ \noalign{\medskip}0&0&0&{\it \psi_{11}}&-{\it \psi_{11}}&-{\it \psi_{12}}
\\ \noalign{\medskip}0&0&-{{\it \psi_{12}}}^{-1}&0&0&0
\\ \noalign{\medskip}{\frac {1+{{\it \psi_{11}}}^{2}}{{\it \psi_{12}}}}&{\it
\psi_{11}}&0&-{\it \psi_{12}}&0&0\\ \noalign{\medskip}-{\frac {1+{{\it \psi_{11}}}
^{2}}{{\it \psi_{12}}}}&-{\it \psi_{11}}&0&0&0&0\\ \noalign{\medskip}-{\it
\psi_{11}}&-{\it \psi_{12}}&0&0&0&0\end {array} \right].
$$

\subsection{The Lie group $G_{3}$}
Let us consider a six-dimensional symplectic Lie group $G_3$ which has a Lie algebra $\g_3$  with non-trivial Lie brackets \cite{Goze-Khakim-Med}:

$[e_1,e_2] = e_3$, $[e_1,e_3] = e_4$,  $[e_1,e_4] = e_5$, $[e_1,e_5] = e_6$.
\vspace{1mm}\\
%$\g_{3}=(0,0,12,13,14,15)$.\\
The sequence of ideals is:
$C^1\g=\R\{e_3,e_4,e_5,e_6\}$,
$C^2\g=\R\{e_4,e_5,e_6\}$,
$C^3\g=\R\{e_5,e_6\}$,
$C^4\g=\R\{e_6\}=\mathcal{Z}$, and the last is the Lie algebra center.
This Lie algebra is filiform and consequently does not suppose an integrable complex structure \cite{Goze-E}.
There is an increasing series of ideals: $\g_1=\mathcal{Z}=\R\{e_6\}$,\, $\g_2=\R\{e_5,e_6\}$,\,
$\g_3=\R\{e_4,e_5,e_6\}$,\, $\g_4=\R\{e_3,e_4,e_5,e_6\}$,\,
$\g_5=\g$. The Lie algebra type is (1,2,3,4,6).

The symplectic structure is defined by the 2-form \cite{Goze-Khakim-Med}:
$$
\omega = e^1\wedge e^6 -e^2\wedge e^5  +e^3\wedge e^4.
$$

We require that the almost complex structure $J$ is $\omega$-compatible and that the ideal $C^1\g$  is $J$-invariant. In this case the ideal $C^3\g$ is also $J$-invariant.
The almost complex structure $J$ is almost nilpotent.
The associated metric will be pseudo-Riemannian.
In addition, we set the free parameter $\psi_{61}$ corresponding to the center $\mathcal{Z}$ to be equal to zero.
The condition $J^2=-1$ gives a complicated expression for the almost complex structure that depends on seven parameters (under the condition that $\psi_{61} =0$). Nevertheless, the curvature tensor $R_{ijk}^s$ is easily calculated, but has many non-zero and complicated components. The scalar square of the curvature tensor $R$ is equal to zero, $g(R,R)=0$.
The scalar curvature $S$ is equal to zero.

The Ricci tensor has a block appearance with one non-zero left upper block of an order 2:
\[
Ric_{2\times 2} =-\frac 12\left[\begin {array}{cc} {{\psi_{11}}}^{2}{{\psi_{12}}}^{2}+{{\psi_{34}}}^{2}& {\psi_{11}}\,{{\psi_{12}}}^{3} \\ \noalign{\medskip}{\psi_{11}}\,{{\psi_{12}}}^{3}&{{\psi_{12}}}^{4}
\end {array} \right]
\]
%$$
%Ric=-\frac 12\left[ \begin {array}{cccccc} {{\psi_{11}}}^{2}{{\psi_{12}}}^{2}+{{\psi_{34}}}^{2}& {\psi_{11}}\,{{\it \psi_{12}}}^{3}&0&0&0&0 \\ \noalign{\medskip}{\psi_{11}}\,{{\psi_{12}}}^{3}&{{\psi_{12}}}^{4}&0&0&0&0\\ \noalign{\medskip}0&0&0&0&0&0
%\\ \noalign{\medskip}0&0&0&0&0&0\\ \noalign{\medskip}0&0&0&0&0&0
%\\ \noalign{\medskip}0&0&0&0&0&0\end {array} \right]
%$$
It depends on three parameters $\psi_{11}$, $\psi_{12}\ne 0$  and $\psi_{34}\ne 0$.
Under the condition $\psi_{34}^2=\psi_{12}^2$ the Ricci tensor is $J$-Hermitian. Then:
\[
Ric_{2\times 2} = -\frac 12\left[ \begin {array}{cc}
{{\psi_{12}}}^{2}(1 +{{\psi_{11}}}^{2})&{\psi_{11}}\,{{ \psi_{12}}}^{3}\\
\noalign{\medskip}{\psi_{11}}\,{{\psi_{12}}}^{3}& {{\psi_{12}}}^{4}
\end {array} \right].
\]
%$$
%Ric=  -\frac 12\left[ \begin {array}{cccccc} {{\psi_{12}}}^{2} +{{\psi_{11}}}^{2}{{\psi_{12}}}^{2}&{\psi_{11}}\,{{\it \psi_{12}}}^{3}&0&0&0
%&0\\
%\noalign{\medskip}{\it \psi_{11}}\,{{\it \psi_{12}}}^{3}& {{
%\it \psi_{12}}}^{4}&0&0&0&0\\ \noalign{\medskip}0&0&0&0&0&0
%\\ \noalign{\medskip}0&0&0&0&0&0\\ \noalign{\medskip}0&0&0&0&0&0
%\\ \noalign{\medskip}0&0&0&0&0&0\end {array} \right].
%$$
Let us assume that all remaining parameters on which the Ricci tensor does not depend are equal to zero. Then we obtain the following expressions for the canonical almost complex structure $J$ and associated almost pseudo-K\"{a}hler metric $g$ with a Hermitian Ricci tensor at $\psi_{34}=\psi_{12}$:
\[
J(e_2)= \psi_{12}\, e_1 -\psi_{11}\, e_2,\qquad
J(e_4)= \psi_{12}\, e_3,\qquad
J(e_6)= \psi_{12}\, e_5 -\psi_{11}\, e_6.
\]
%$$
%J= \left[ \begin {array}{cccccc} {\it \psi_{11}}&{\it \psi_{12}}&0&0&0&0\\ \noalign{\medskip}-{\frac {1+{{\it \psi_{11}}}^{2}}{{\it \psi_{12}}}}&-{
%\it \psi_{11}}&0&0&0&0\\ \noalign{\medskip}0&0&0&{\it \psi_{12}}&0&0
%\\ \noalign{\medskip}0&0&-{{\it \psi_{12}}}^{-1}&0&0&0
%\\ \noalign{\medskip}0&0&0&0&{\it \psi_{11}}&{\it \psi_{12}}
%\\ \noalign{\medskip}0&0&0&0&-{\frac {1+{{\it \psi_{11}}}^{2}}{{\it \psi_{12}}
%}}&-{\it \psi_{11}}\end {array} \right],
%$$
$$
g= -\left[ \begin {array}{cccccc} 0&0&0&0& {\frac {1+{{\it \psi_{11}}}^{2}}{{
\it \psi_{12}}}}& {\it \psi_{11}}\\
\noalign{\medskip}0&0&0&0& {\it \psi_{11}}& {
\it \psi_{12}}\\ \noalign{\medskip}0&0& {{\it \psi_{12}}}^{-1}&0&0&0
\\ \noalign{\medskip}0&0&0& {\it \psi_{12}}&0&0\\ \noalign{\medskip} {
\frac {1+{{\it \psi_{11}}}^{2}}{{\it \psi_{12}}}}& {\it \psi_{11}}&0&0&0&0
\\ \noalign{\medskip} {\it \psi_{11}}& {\it \psi_{12}}&0&0&0&0\end {array}
 \right].
$$

\subsection{The Lie group $G_{4}$}
%замена e4 == e5.
Let us consider a six-dimensional symplectic Lie group $G_4$ which has a Lie algebra $\g_4$  with non-trivial Lie brackets \cite{Goze-Khakim-Med}:

%$[e_1,e_2] = e_3$, $[e_1,e_3] = e_4$,  $[e_1,e_4] = e_6$, $[e_2,e_3] = e_5$, $[e_2,e_5] = %e_6$.
$[e_1,e_2] = e_3$, $[e_1,e_3] = e_5$,  $[e_1,e_5] = e_6$, $[e_2,e_3] = e_4$, $[e_2,e_4] = e_6$.
\vspace{1mm}\\
%$\g_{4}=(0,0,12,13,23,14+25)$.\\
This group supposes separately symplectic and complex structures, but does not suppose compatible complex structures.

The sequence of ideals is:
$C^1\g=\R\{e_3,e_4,e_5,e_6\}$,
$C^2\g=\R\{e_4,e_5,e_6\}$,
$C^3\g=\R\{e_6\}=\mathcal{Z}$, and the last is the Lie algebra center.
There is an increasing series of ideals: $\g_1=\mathcal{Z}=\R\{e_6\}$,\, $\g_2=\R\{e_4,e_5,e_6\}$,\,
$\g_3=\R\{e_3,e_4,e_5,e_6\}$,\,
$\g_4=\g$. The Lie algebra type is (1,3,4,6).

The symplectic structure is defined by the 2-form \cite{Goze-Khakim-Med}:
$$
\omega = e^1\wedge e^4 + t\,e^1\wedge e^5 +e^1\wedge e^6 +e^2\wedge e^5 +e^3\wedge e^4.
$$

We require that the almost complex structure $J$ is $\omega$-compatible and that the ideal $C^1\g$  is $J$-invariant.
In this case the ideal $\R\{e_5,e_6\}$ is also $J$-invariant.
The almost complex structure $J$ is almost nilpotent.
The associated metric will be pseudo-Riemannian.
In addition, we set the free parameters corresponding to the center $\mathcal{Z}$ to be equal to zero, $\psi_{61}=0$.
The condition $J^2=-1$ gives a complicated expression for the almost complex structure that depends on seven parameters (under the condition that $\psi_{61} =0$). Nevertheless, the curvature tensor $R_{ijk}^s$ is easily calculated, but has many non-zero and complicated components. The scalar square of the curvature tensor $R$ is equal to zero, $g(R,R)=0$.
The scalar curvature $S$ is equal to zero.

The Ricci tensor has a block appearance with one non-zero left upper block of an order 2:
\[
Ric_{2\times 2} = -\frac 12 \left[ \begin {array}{cc}
{{\psi_{12}}}^{2}{{\psi_{11}}}^{2}& {\psi_{11}}\,{{\psi_{12}}}^{3}\\ \noalign{\medskip}{\psi_{11}}\,{{\psi_{12}}}^{3}& {{\psi_{12}}}^{4}+{{\psi_{34}}}^{2}
\end {array} \right].
\]
%$$
%Ric= -\frac 12 \left[ \begin {array}{cccccc} {{\psi_{12}}}^{2}{{\psi_{11}}}^{2}& {\psi_{11}}\,{{\psi_{12}}}^{3}&0&0&0&0\\ \noalign{\medskip}{\psi_{11}}\,{{\psi_{12}}}^{3}& {{\psi_{12}}}^{4}+{{\psi_{34}}}^{2}&0&0&0&0\\
%\noalign{\medskip}0&0&0&0&0&0
%\\ \noalign{\medskip}0&0&0&0&0&0\\ \noalign{\medskip}0&0&0&0&0&0
%\\ \noalign{\medskip}0&0&0&0&0&0\end {array} \right].
%$$
It depends on three parameters $\psi_{11}$, $\psi_{12}\ne 0$  and $\psi_{34}\ne 0$.
The Ricci tensor is not $J$-Hermitian for any values of the parameters.

Let us assume that all remaining parameters on which the Ricci tensor does not depend are equal to zero. Then we obtain the following expressions for the canonical almost complex structure $J$ and associated almost pseudo-K\"{a}hler metric $g$:
\[
J=\left[ \begin {array}{cccccc} {\it \psi_{11}}&{\it \psi_{12}}&0&0&0&0
\\ \noalign{\medskip}-{\frac {1+{{\it \psi_{11}}}^{2}}{{\it \psi_{12}}}}&-{
\it \psi_{11}}&0&0&0&0\\ \noalign{\medskip}0&0&0&{\psi_{34}}&0&0\\
\noalign{\medskip}0&0&-{{\it \psi_{34}}}^{-1}&0&0&0
\\ \noalign{\medskip}0&0&0&-{\it \psi_{12}}&-{\it \psi_{12}}\,t+{\it \psi_{11}}&-{
\it \psi_{12}}\\ \noalign{\medskip}0&0&{{\it \psi_{34}}}^{-1}&-{\it \psi_{11}}+{
\it \psi_{12}}\,t&{\frac {-2\,{\it \psi_{12}}\,{\it \psi_{11}}\,t+1+{{\it \psi_{11}}}^
{2}+{{\it \psi_{12}}}^{2}{t}^{2}}{{\it \psi_{12}}}}&-{\it \psi_{11}}+{\it \psi_{12}}\,
t\end {array} \right]
\]
\[
g=\left[ \begin {array}{cccccc}
0&0&0&-{\psi_{11}}&-{\frac {\psi_{12}
\,\psi_{11}\,t-1-{{\psi_{11}}}^{2}}{\psi_{12}}}&-{\psi_{11}}\\
\noalign{\medskip}0&0&0&-\psi_{12}&-\psi_{12}\,t+{\psi_{11}}&-\psi_{12}\\
\noalign{\medskip}0&0&-{\psi_{34}}^{-1}&0&0&0\\
\noalign{\medskip}-{\psi_{11}}&-\psi_{12}&0&-\psi_{34}&0&0\\
\noalign{\medskip}-{\frac {\psi_{12}\,{\psi_{11}}\,t-1-{\psi_{11}}^{2}}{\psi_{12}}}&-\psi_{12}\,t+\psi_{11}&0&0&0&0\\
\noalign{\medskip}-\psi_{11}&-\psi_{12}&0&0&0&0
\end {array} \right]
\]

\subsection{The Lie group $G_{5}$}
Let us consider a six-dimensional symplectic Lie group $G_5$ which has a Lie algebra $\g_5$  with non-trivial Lie brackets \cite{Goze-Khakim-Med}:

$[e_1,e_2] = e_3$, $[e_1,e_3] = e_4$,  $[e_1,e_4] = -e_6$, $[e_2,e_3] = e_5$, $[e_2,e_5] = e_6$.\\
%$\g_{5}=(0,0,12,13,23,25-14)$.\\
The sequence of ideals is:
$C^1\g=\R\{e_3,e_4,e_5,e_6\}$,
$C^2\g=\R\{e_4,e_5,e_6\}$,
$C^3\g=\R\{e_6\}=\mathcal{Z}$, and the last is the Lie algebra center.
There is an increasing series of ideals: $\g_1=\mathcal{Z}=\R\{e_6\}$,\, $\g_2=\R\{e_4,e_5,e_6\}$,\,
$\g_3=\R\{e_3,e_4,e_5,e_6\}$,\,
$\g_4=\g$. The Lie algebra type is (1,3,4,6).

The Lie group $G_5$ has 4 left-invariant symplectic structures  \cite{Goze-Khakim-Med}.

\textbf{First case.} The symplectic structure is:
$$
\omega_1 = t\,e^1\wedge e^4 +e^1\wedge e^5 +e^1\wedge e^6 +e^2\wedge e^4 +e^3\wedge e^5.
$$

We require that the almost complex structure $J$ is $\omega_1$-compatible and that the ideal $C^1\g$ is $J$-invariant. In this case the ideal $\R\{e_4,e_6\}$ is also $J$-invariant. The almost complex structure $J$ is almost nilpotent.
The associated metric will be pseudo-Riemannian.
The condition $J^2=-1$ gives a complicated expression for the almost complex structure that depends on seven parameters (under the condition that $\psi_{61}=0$). Nevertheless, the curvature tensor $R_{ijk}^s$ is easily calculated, but has many non-zero and complicated components. The scalar square $g(R,R)=0$ of the curvature tensor $R$ is equal to zero. The scalar curvature $S$ is equal to zero.

The Ricci tensor has a block appearance with one non-zero left upper block of an order 2.
\[
Ric_{2\times 2} =-\frac 12\left[ \begin {array}{cc} {{\psi_{12}}}^{2}{{\psi_{11}}}^{2
}& {\psi_{11}}\,{{\psi_{12}}}^{3}\\
\noalign{\medskip}{\psi_{11}}\,{{\psi_{12}}}^{3}& {{\psi_{12}}}^{4}+{{\psi_{35}}}^{2}
\end {array} \right].
\]
%$$
%Ric= -\frac 12\left[ \begin {array}{cccccc} {{\psi_{12}}}^{2}{{\psi_{11}}}^{2
%}& {\psi_{11}}\,{{\psi_{12}}}^{3}&0&0&0&0\\
%\noalign{\medskip}{\psi_{11}}\,{{\psi_{12}}}^{3}& {{\psi_{12}}}^{4}+{{\psi_{35}}}^{2}&0&0&0&0\\ \noalign{\medskip}0&0&0&0&0&0
%\\ \noalign{\medskip}0&0&0&0&0&0\\ \noalign{\medskip}0&0&0&0&0&0
%\\ \noalign{\medskip}0&0&0&0&0&0\end {array} \right].
%$$
The Ricci tensor is not $J$-Hermitian for any values of the parameters.

Let us assume that all remaining parameters on which the Ricci tensor does not depend are equal to zero. Then we obtain the following expressions for the canonical almost complex structure and associated almost pseudo-K\"{a}hler metric:
$$
J=  \left[ \begin {array}{cccccc} {\psi_{11}}&{\psi_{12}}&0&0&0&0
\\ \noalign{\medskip}-{\frac {1+{{\psi_{11}}}^{2}}{{\psi_{12}}}}&-{
\psi_{11}}&0&0&0&0\\ \noalign{\medskip}0&0&0&0&{\psi_{35}}&0
\\ \noalign{\medskip}0&0&0&-{\psi_{12}}\,t+{\psi_{11}}&-{\psi_{12}}&-{
\psi_{12}}\\ \noalign{\medskip}0&0&-{{\psi_{35}}}^{-1}&0&0&0
\\
\noalign{\medskip}0&0&{{\psi_{35}}}^{-1}&J_{64}&-{\psi_{11}}+{\psi_{12}}\,t&-{\psi_{11}}+{\psi_{12}}\,t
\end {array} \right],
$$
where $J_{64}= {\frac {-2\,{\psi_{12}}\,{
\psi_{11}}\,t+{{\psi_{12}}}^{2}{t}^{2}+1+{{\psi_{11}}}^{2}}{{\psi_{12}
}}}$,
$$
g=  -\left[ \begin {array}{cccccc} 0&0&0&{\frac {{\psi_{12}}\,{\psi_{11}}
\,t-1-{{\psi_{11}}}^{2}}{{\psi_{12}}}}&{\psi_{11}}&{\psi_{11}}\\  \noalign{\medskip}0&0&0&{\psi_{12}}\,t-{\psi_{11}}&{\psi_{12}}&{\psi_{12}}\\ \noalign{\medskip}0&0&{{\psi_{35}}}^{-1}&0&0&0\\
\noalign{\medskip}{\frac{{\psi_{12}}\,{\psi_{11}}\,t -1-{{\psi_{11}}}^{2}}{{\psi_{12}}}}&{\psi_{12}}\,t-{\psi_{11}}&0&0&0&0 \\ \noalign{\medskip}{\psi_{11}}&{\psi_{12}}&0&0&{\psi_{35}}&0\\ \noalign{\medskip}{\psi_{11}}&{\psi_{12}}&0&0&0&0\end {array}
 \right].
$$

\textbf{Second case.} The symplectic structure is:
$$
\omega_2 = e^1\wedge e^6 -2\,e^1\wedge e^5 -2\,e^2\wedge e^4 +e^2\wedge e^6 +e^3\wedge e^4 +e^3\wedge e^5.
$$

We require that the ideal $C^1\g$ is invariant for $J$. The associated metric will be pseudo-Riemannian. Free parameters among $\psi_{6i}$ are not present.
The condition $J^2=-1$ gives a complicated expression for the almost complex structure that depends on eight parameters. Nevertheless, the curvature tensor $R_{ijk}^s$ is easily calculated, but has many non-zero and complicated components. The scalar square $g(R,R)=0$ of the curvature tensor $R$ is equal to zero. The scalar curvature $S$ is equal to zero.

The Ricci tensor has the non-zero left upper ($2\times 2$)-block and depends on tree parameters $\psi_{12}$, $\psi_{34}$ and $\psi_{11}$.
The Ricci tensor remains too complicated. Let us assume in addition that $\psi_{11}=0$.
Under the condition that $\psi_{34}=\frac{1+\psi_{12}^2}{4\psi_{12}}$
the Ricci tensor is $J$-Hermitian, $Ric(JX,JY)=Ric(X,Y)$. Then we obtain:
\[
Ric_{2\times 2} = -\frac {1}{32}\left[ \begin {array}{cc}
{\frac {1+3\,{{\psi_{12}}}^{2}+{
{\psi_{12}}}^{6}+3\,{{\psi_{12}}}^{4}}{{{\psi_{12}}}^{4}}}\\
\noalign{\medskip}0&{\frac {1+3\,{{\psi_{12}}}^{2}+{{\psi_{12}}}^{6}+ 3\,{{\psi_{12}}}^{4}}{{{\psi_{12}}}^{2}}}
\end {array} \right].
\]
%$$
%Ric= -\frac {1}{32}\left[ \begin {array}{cccccc} {\frac {1+3\,{{\psi_{12}}}^{2}+{
%{\psi_{12}}}^{6}+3\,{{\psi_{12}}}^{4}}{{{\psi_{12}}}^{4}}}&0&0&0&0&0
%\\ \noalign{\medskip}0&{\frac {1+3\,{{\psi_{12}}}^{2}+{{\psi_{12}}}^{6}+3\,{{\psi_{12}}}^{4}}{{{\psi_{12}}}^{2}}}&0&0&0&0
%\\ \noalign{\medskip}0&0&0&0&0&0\\ \noalign{\medskip}0&0&0&0&0&0
%\\ \noalign{\medskip}0&0&0&0&0&0\\ \noalign{\medskip}0&0&0&0&0&0
%\end {array} \right].
%$$
Let us assume that all remaining parameters on which the Ricci tensor does not depend are equal to zero. Then we obtain the following expressions for the canonical almost complex structure and associated almost pseudo-K\"{a}hler metric with a Hermitian Ricci tensor:
$$
J= \left[ \begin {array}{cccccc} 0&{\psi_{12}}&0&0&0&0\\
 \noalign{\medskip}-{{\psi_{12}}}^{-1}&0&0&0&0&0\\
 \noalign{\medskip}-2\,{{\psi_{12}}}^{-1}&0&0&{\frac {{{\psi_{12}}}^{2}+1}{4\,{\psi_{12}}}}&{\frac {{{\psi_{12}}}^{2}+1}{4\,{\psi_{12}}}}&0\\
 \noalign{\medskip}0&0&-{\frac {4\,{\psi_{12}}}{{{\psi_{12}}}^{2}+1}}&0&{\frac {1-{{\psi_{12}}}^{2}}{2\,{\psi_{12}}}}&{\frac {{{\psi_{12}}}^{2}+1}{4\,{\psi_{12}}}}\\
 \noalign{\medskip}0&{\frac {8\,{\psi_{12}}}{{{\psi_{12}}}^{2} +1}}&0&0& {\frac {-1+{{\psi_{12}}}^{2}}{2\,{\psi_{12}}}}&-{\frac {{{\psi_{12}}}^{2}+1}{4\,{
\psi_{12}}}}\\
 \noalign{\medskip}-\frac {32\,{\psi_{12}}}{(1+{\psi_{12}}^2)^2}& \frac {16\, \left( -1+{{\psi_{12}}
}^{2} \right) {\psi_{12}}}{(1+{\psi_{12}}^2)^2}&0
&0&{\frac {{{\psi_{12}}}^{2}+1}{{\psi_{12}}}}&{\frac {1-{{\psi_{12}}}^{2}}{2\,{\psi_{12}}}}\end {array} \right],
$$
$$
g= \left[ \begin {array}{cccccc}
-\frac {32\,{\psi_{12}}}{(1+{\psi_{12}}^2)^2}&-\frac {32\,{\psi_{12}}}{(1+{\psi_{12}}^2)^2}&0&0&2\,{{\psi_{12}}}^{-1}&{{\psi_{12}}}^{-1}\\
\noalign{\medskip}-\frac {32\,{\psi_{12}}}{(1+{\psi_{12}}^2)^2}& \frac {16\, \left( {{\psi_{12}}}^{2}-1 \right) {\psi_{12}}}{(1+{\psi_{12}}^2)^2}& {
\frac {8\,{\psi_{12}}}{1+{{\psi_{12}}}^{2}}}&0&2\,{\psi_{12}}&-{\psi_{12}
}\\
\noalign{\medskip}0&{\frac {8\,{\psi_{12}}}{1+{{\psi_{12}}}^{2}}}&
-{\frac {4\,{\psi_{12}}}{1+{{\psi_{12}}}^{2}}}&0&0&0 \\
 \noalign{\medskip}0&0&0&-{\frac{1+{{\psi_{12}}}^{2}}{4\,{\psi_{12}}}}& -{\frac {1+{{\psi_{12}}}^{2}}{4\,{\psi_{12}}}}&0\\
  \noalign{\medskip}2\,{{\psi_{12}}}^{-1}&2\,{\psi_{12}}&0&-{
\frac {1+{{\psi_{12}}}^{2}}{4\,{\psi_{12}}}}&-{\frac {1+{{\psi_{12}}
}^{2}}{4\,{\psi_{12}}}}&0\\ \noalign{\medskip}{{\psi_{12}}}^{-1}&-{\psi_{12}}&0&0&0&0\end {array} \right].
$$

\textbf{Third case.} The symplectic structure is:
$$
\omega_3 = e^1\wedge e^4 -e^1\wedge e^5 +e^1\wedge e^6 -e^2\wedge e^4 +e^2\wedge e^5 +e^2\wedge e^6 +e^3\wedge e^4 +e^3\wedge e^5.
$$

We require that the ideal $C^1\g$ is invariant for $J$. The associated metric will be pseudo-Riemannian. Free parameters among $\psi_{6i}$ are not present.
The condition $J^2=-1$ gives a complicated expression for the almost complex structure that depends on eight parameters. Nevertheless, the curvature tensor $R_{ijk}^s$ is easily calculated, but has many non-zero and complicated components. The scalar square $g(R,R)=0$ of the curvature tensor $R$ is equal to zero. The scalar curvature $S$ is equal to zero.

The Ricci tensor has the non-zero left upper ($2\times 2$)-block and depends on three parameters $\psi_{11}$, $\psi_{12}\ne 0$ and $\psi_{34}$. It is $J$-Hermitian at $\psi_{34}=\pm\frac{(\psi_{11}-\psi_{12})^2+1}{4\psi_{12}}$. We take for clarity the positive value $\psi_{34}=\frac{(\psi_{11}-\psi_{12})^2+1}{4\psi_{12}}$. The Ricci tensor remains too complicated. Let us assume in addition that $\psi_{11}=0$. Then:
\[
Ric_{2\times 2} =\left[ \begin {array}{cc}
-\,\frac {(1+{\psi_{12}}^2)^3}{32\,{\psi_{12}}^{4}}\\
\noalign{\medskip}0&-\,\frac {(1+{\psi_{12}}^2)^3}{32\,{\psi_{12}}^{2}}
\end {array} \right],
\]
%$$
%Ric= \left[ \begin {array}{cccccc} -\,\frac {(1+{\psi_{12}}^2)^3}{32\,{\psi_{12}}^{4}}&0&0&0&0&0
%\\ \noalign{\medskip}0&-\,\frac {(1+{\psi_{12}}^2)^3}{32\,{\psi_{12}}^{2}}&0&0&0&0
%\\ \noalign{\medskip}0&0&0&0&0&0\\ \noalign{\medskip}0&0&0&0&0&0
%\\ \noalign{\medskip}0&0&0&0&0&0\\ \noalign{\medskip}0&0&0&0&0&0
%\end {array} \right],
%$$

Let us assume that all remaining parameters on which the Ricci tensor does not depend are equal to zero. Then we obtain the following expressions for the canonical almost complex structure and associated almost pseudo-K\"{a}hler metric with a Hermitian Ricci tensor:
$$
J=\left[ \begin {array}{cccccc} 0&\psi_{12}&0&0&0&0\\
 \noalign{\medskip}-\frac{1}{\psi_{12}}&0&0&0&0&0\\
 \noalign{\medskip}0&-{\frac {{\psi_{12}}^{2}+1}{\psi_{12}}}&
 \frac{1}{\psi_{12}}&{\frac {{\psi_{12}}^{2}+1}{4\,\psi_{12}}}&
{\frac {{\psi_{12}}^{2}+1}{4\,\psi_{12}}}&0\\
 \noalign{\medskip}0&0&-\frac{4}{\psi_{12}}& -\frac{1}{\psi_{12}}&
 -1{\frac {{\psi_{12}}^{2}+1}{2\,\psi_{12}}}&{\frac {{\psi_{12}}^{2}+1}{4\,\psi_{12}}}\\
 \noalign{\medskip}-\frac{4}{\psi_{12}}&\frac{4}{\psi_{12}}&0&0&
{\frac {{\psi_{12}}^{2}-1}{2\,\psi_{12}}}&-{\frac {{\psi_{12}}^{2}+1}{4\,\psi_{12}}}\\
 \noalign{\medskip}-\frac{8}{\psi_{12}}&-\frac{8}{\psi_{12}}&0&0&{\frac {{\psi_{12}}^{2}+1}{\psi_{12}}
}&-{\frac {{\psi_{12}}^{2}-1}{2\,\psi_{12}}}\end {array}
 \right],
$$
$$
g=-\frac{1}{\psi_{12}}\left[ \begin {array}{cccccc} 4&12&4&1&-1&-1\\ \noalign{\medskip}12&4&-4&-1&-1-2\,{{\psi_{12}}}^{2}&{{\psi_{12}}}^{2}\\ \noalign{\medskip}4&-4&4&1&1&0\\
 \noalign{\medskip}1&-1&1&({{\psi_{12}}}^{2}+1)/4&({{\psi_{12}}}^{2}+1)/4&0\\
 \noalign{\medskip}-1&-1-2\,{{\psi_{12}}}^{2}& 1& ({{\psi_{12}}}^{2}+1)/4 & ({{\psi_{12}}}^{2}+1)/4&0\\
 \noalign{\medskip}-1&{{\psi_{12}}}^{2}&0&0&0&0\end {array} \right].
$$

\textbf{Fourth case.} The symplectic structure is:
$$
\omega_4 = 2\,e^1\wedge e^4 +e^1\wedge e^6 +2\,e^2\wedge e^5 +e^2\wedge e^6 +e^3\wedge e^4 +e^3\wedge e^5.
$$

We require that the ideal $C^1\g$ is invariant for $J$. The associated metric will be pseudo-Riemannian. Free parameters among $\psi_{6i}$ are not present.
The condition $J^2=-1$ gives a complicated expression for the almost complex structure that depends on eight parameters. Nevertheless, the curvature tensor $R_{ijk}^s$ is easily calculated, but has many non-zero and complicated components. The scalar square $g(R,R)=0$ of the curvature tensor $R$ is equal to zero. The scalar curvature $S$ is equal to zero.

The Ricci tensor has the non-zero left upper ($2\times 2$)-block and depends on three parameters $\psi_{11}$, $\psi_{12}\ne 0$ and $\psi_{34}$. It is $J$-Hermitian at $\psi_{34}=\pm\frac{(\psi_{11}-\psi_{12})^2+1}{4\psi_{12}}$. We take for clarity the positive value $\psi_{34}=\frac{(\psi_{11}-\psi_{12})^2+1}{4\psi_{12}}$. The Ricci tensor remains too complicated. Let us assume in addition that $\psi_{11}=0$. Then:
\[
Ric_{2\times 2} =\left[ \begin {array}{cc} -\,\frac {(1+{\psi_{12}}^2)^3}{32\,{\psi_{12}}^{4}}\\
\noalign{\medskip}0&-\,\frac {(1+{\psi_{12}}^2)^3}{32\,{\psi_{12}}^{2}}
\end {array} \right],
\]
%$$
%Ric= \left[ \begin {array}{cccccc} -\,\frac {(1+{\psi_{12}}^2)^3}{32\,{\psi_{12}}^{4}}&0&0&0&0&0
%\\ \noalign{\medskip}0&-\,\frac {(1+{\psi_{12}}^2)^3}{32\,{\psi_{12}}^{2}}&0&0&0&0
%\\ \noalign{\medskip}0&0&0&0&0&0\\ \noalign{\medskip}0&0&0&0&0&0
%\\ \noalign{\medskip}0&0&0&0&0&0\\ \noalign{\medskip}0&0&0&0&0&0
%\end {array} \right],
%$$
%$$
Let us assume that all remaining parameters on which the Ricci tensor does not depend are equal to zero. Then we obtain the following expressions for the canonical almost complex structure and associated almost pseudo-K\"{a}hler metric with a Hermitian Ricci tensor:
$$
J=\left[ \begin {array}{cccccc} 0&{\psi_{12}}&0&0&0&0\\ -\frac{1}{\psi_{12}}&0&0&0&0&0\\
2\,{\psi_{12}}&0&0& {\frac {{{\psi_{12}}}^{2} +1}{4\,{\psi_{12}}}}&{\frac {{{\psi_{12}}}^{2}+1}{4\,{\psi_{12}}}}&0 \\
0&0&{\frac {2\,{\psi_{12}}\, \left( {{\psi_{12}}}
^{2}-1 \right) }{{{\psi_{12}}}^{2}+1}}&0&-{\frac {{{\psi_{12}}}^{
2}-1}{2\,{\psi_{12}}}}&{\frac {{{\psi_{12}}}^{2}+1}{4\,{\psi_{12}}}}\\
0&-{\frac {8\,{{\psi_{12}}}^{3}}{{{\psi_{12}}}^{2}+1}}& -2\,{\psi_{12}}& 0& {\frac {{{\psi_{12}}}^{2}-1}{2\,{\psi_{12}}}
}&-{\frac {{{\psi_{12}}}^{2}+1}{4\,{\psi_{12}}}} \\
J_{61}&J_{62}&-{\frac {4\,{\psi_{12}}\, \left( {
{\psi_{12}}}^{2}-1 \right) }{{{\psi_{12}}}^{2}+1}}&-2\,{\psi_{12}}&
{\frac {1-{{\psi_{12}}}^{2}}{{\psi_{12}}}}&{\frac {1-{{\psi_{12}}}^
{2}}{2\,{\psi_{12}}}}\end {array} \right],
$$
where $J_{61}=J_{62}=-\frac {16\,{{\psi_{12}}}^{3} \left( {{\psi_{12}
}}^{2}-1 \right) }{({\psi_{12}}^2+1)^2}$,
$$
g=\left[ \begin {array}{cccccc} -\frac {16\,{{\psi_{12}}}^{3} \left( {{\psi_{12}}}^{2}-1 \right) }{({\psi_{12}}^2+1)^2}
& -\frac {16\,{{\psi_{12}}}^{3} \left( {{\psi_{12}}}^{2}-1 \right) }{({\psi_{12}}^2+1)^2}&0&-2\,{\psi_{12}}&-{\frac {2\,{{\psi_{12}}}^{2}-1}{{\psi_{12}}}}&\frac{1}{\psi_{12}}\\
 \noalign{\medskip}-\frac {16\,{{\psi_{12}}}^{3} \left( {{\psi_{12}
}}^{2}-1 \right) }{({\psi_{12}}^2+1)^2}& -{
\frac {32\,{{\psi_{12}}}^{5}}{ \left( {{\psi_{12}}}^{2}+1 \right) ^{2}}}&
-{\frac {8\,{{\psi_{12}}}^{3}}{{{\psi_{12}}}^{2}+1}}&-2\,{\psi_{12}}&0&
-{\psi_{12}}\\
 \noalign{\medskip}0&-{\frac {8\,{{\psi_{12}}}^{3}}{{{
\psi_{12}}}^{2}+1}}&-{\frac {4\,{\psi_{12}}}{{{\psi_{12}}}^{2}+1}}&0&0
&0\\
 \noalign{\medskip}-2\,{\psi_{12}}&-2\,{\psi_{12}}&0&-{\frac
{{{\psi_{12}}}^{2}+1}{4\,{\psi_{12}}}}&-{\frac {{{\psi_{12}}}^{2}+1}
{4\,{\psi_{12}}}}&0\\
 \noalign{\medskip}-2\,{\frac {{{\psi_{12}}}^{2}-1}{
{\psi_{12}}}}&0&0&-{\frac {{{\psi_{12}}}^{2}+1}{4\,{\psi_{12}}}}&-{\frac {{{\psi_{12}}}^{2}+1}{4\,{\psi_{12}}}}&0\\
 \noalign{\medskip}\frac{1}{\psi_{12}}&-{\psi_{12}}&0&0&0&0\end {array} \right].
$$

\subsection{The Lie group $G_{6}$}
Let us consider a six-dimensional symplectic Lie group $G_6$ which has a Lie algebra $\g_6$  with non-trivial Lie brackets \cite{Goze-Khakim-Med}:

$[e_1,e_2] = e_3$, $[e_1,e_3] = e_4$,  $[e_1,e_4] = e_5$, $[e_2,e_3] = e_6$.\\
%$\g_{6}=(0,0,12,13,14,23)$.\\
The sequence of ideals is:
$C^1\g=\R\{e_3,e_4,e_5,e_6\}$,
$C^2\g=\R\{e_4,e_5,e_6\}$,
$C^3\g=\R\{e_5\}$,
$\mathcal{Z}=\R\{e_5,e_6\}$, and the last is the Lie algebra center.
There is an increasing series of ideals: $\g_1=\mathcal{Z}=\R\{e_5,e_6\}$,\, $\g_2=\R\{e_4,e_5,e_6\}$,\,
$\g_3=\R\{e_3,e_4,e_5,e_6\}$,\,
$\g_4=\g$. The Lie algebra type is (2,3,4,6).

The symplectic structure is defined by the 2-form \cite{Goze-Khakim-Med}:
$$
\omega = e^1\wedge e^6 +e^2\wedge e^4 +e^2\wedge e^5 -e^3\wedge e^4,
$$

We require of the compatible almost complex structure that the ideal $C^1\g$ is $J$-invariant. Then the center $\mathcal{Z}$ will also be $J$-invariant. The almost complex structure $J$ is almost nilpotent. In this case the associated metric will be pseudo-Riemannian.
As $\mathcal{Z}=\R\{e_5,e_6\}$, then according to Lemma \ref{ParRiem} and Corollary \ref{ParCurvTensor}, it is possible to assume that some of the free parameters ($\psi_{ij}$, $i=5,6,\, j=1,\dots, 6$) are equal to zero (the curvature does not depend on them). Such free parameters will be $\psi_{51}$, $\psi_{52}$ and $\psi_{61}$. We assume that they are zero.

The condition $J^2=-1$ gives a complicated expression for the almost complex structure that depends on eight parameters. Nevertheless, the curvature tensor $R_{ijk}^s$ is easily calculated, but has many non-zero and complicated components. The scalar square $g(R,R)=0$ of the curvature tensor $R$ is equal to zero. The scalar curvature $S$ is equal to zero.

The Ricci tensor has a block appearance with one non-zero left upper block of an order 1:
$$
Ric=-\frac {(1+{{\psi_{33}}}^{2})^2}{2\,{{\psi_{43}}}^{2}}\, (e^1)^2.
$$
%$$
%Ric=-\frac 12\,\left[ \begin {array}{cccccc} {\frac { \left( 1+{{\psi_{33}}}^{2} \right) ^{2}}{{{\psi_{43}}}^{2}}}&0&0&0&0&0\\ \noalign{\medskip}0&0
%&0&0&0&0\\ \noalign{\medskip}0&0&0&0&0&0\\ \noalign{\medskip}0&0&0&0&0
%&0\\ \noalign{\medskip}0&0&0&0&0&0\\ \noalign{\medskip}0&0&0&0&0&0
%\end {array} \right].
%$$
It depends on two parameters $\psi_{33}$ and $\psi_{43}\ne 0$.
The Ricci tensor is not $J$-Hermitian for any values of the parameters.

Let us set all the remaining parameters on which the tensor of Ricci does not depend to be zero. Then we obtain the following expressions for the canonical almost complex structure and associated almost pseudo-K\"{a}hler metric:
$$
J=\left[ \begin {array}{cccccc} 0&1&0&0&0&0\\ \noalign{\medskip}-1&0&0&0&0&0\\ \noalign{\medskip}0&0&{\psi_{33}}&-{\frac {1+{{\psi_{33}}}^{2}}
{{\psi_{43}}}}&0&0\\ \noalign{\medskip}0&0&{\psi_{43}}&-{\psi_{33}}&0&0
\\ \noalign{\medskip}0&0&-{\psi_{43}}&{\psi_{33}}&0&-1
\\ \noalign{\medskip}0&0&0&1&1&0\end {array} \right],
$$
$$
g=\left[ \begin {array}{cccccc} 0&0&0&1&1&0\\ \noalign{\medskip}0&0&0&0&0&-1\\ \noalign{\medskip}0&0&-{\psi_{43}}&{\psi_{33}}&0&0
\\ \noalign{\medskip}1&0&{\psi_{33}}&-{\frac {1+{{\psi_{33}}}^{2}}{{
\psi_{43}}}}&0&0\\ \noalign{\medskip}1&0&0&0&0&0\\ \noalign{\medskip}0
&-1&0&0&0&0\end {array} \right].
$$

%%%%%%%%%%%%%%%%%%%%%%%%%

\subsection{The Lie group $G_{7}$}
Let us consider a six-dimensional symplectic Lie group $G_7$ which has a Lie algebra $\g_7$  with non-trivial Lie brackets \cite{Goze-Khakim-Med}:

$[e_1,e_2] = e_4$, $[e_1,e_4] = e_5$,  $[e_1,e_5] = e_6$, $[e_2,e_3] = e_6$, $[e_2,e_4] = e_6$.\\
This Lie algebra is semidirect product of a five-dimensional Lie algebra $\mathfrak{h}= \R\{e_1,e_2,e_4,e_5,e_6\}$ and $\mathbb{R}^1=\mathbb{R}\{e_3\}$
%$\g_{7}=(0,0,0,12,14,15+23+24)$.\\
The sequence of ideals is:
$C^1\g=\R\{e_4,e_5,e_6\}$,
$C^2\g=\R\{e_5,e_6\}$,
$C^3\g=\R\{e_6\}=\mathcal{Z}$, and the last is the Lie algebra center.
There is an increasing series of ideals: $\g_1=\mathcal{Z}=\R\{e_6\}$,\, $\g_2=\R\{e_3,e_5,e_6\}$,\,
$\g_3=\R\{e_3,e_4,e_5,e_6\}$,\,
$\g_4=\g$. The Lie algebra type is (1,3,4,6).

The Lie group $G_7$ has 2 left-invariant symplectic structures  \cite{Goze-Khakim-Med}.

\textbf{First case.} The symplectic structure is:
$$
\omega_1 = e^1\wedge e^3 +e^2\wedge e^6 -e^4\wedge e^5.
$$

We require that the almost complex structure $J$ is $\omega_1$-compatible and that the ideal $\g_3$ is $J$-invariant. In this case the ideal $I=\R\{e_3,e_6\}$ is also $J$-invariant. The almost complex structure $J$ is almost nilpotent. The associated metric will be pseudo-Riemannian. We suppose also that $\psi_{61}=0$.

The condition $J^2=-1$ gives a complicated expression for the almost complex structure that depends on seven parameters (under the condition that $\psi_{61} =0$). Nevertheless, the curvature tensor is easily calculated, but has many non-zero and complicated components. The scalar square of the curvature tensor $R$ is equal to zero, $g(R,R)=0$. The scalar curvature $S$ is equal to zero.

The Ricci tensor has a block appearance with one non-zero left upper block of an order 2:
\[
Ric_{2\times 2}=-\frac 12 \left[ \begin {array}{cc}
{\frac {\left( 1+{\psi_{11}}^{2}\right)^{4}+{\psi_{45}}^{2}{\psi_{12}}^{4}}{{\psi_{12}}^{4}}}&{\frac { \left( 1+{\psi_{11}}^{2} \right) ^{3}\psi_{11}}{{\psi_{12}}^{3}}}\\
\noalign{\medskip}{\frac { \left( 1+{\psi_{11}}^{2}
 \right)^{3}\psi_{11}}{{\psi_{12}}^{3}}}&{\frac {\left(1+
{\psi_{11}}^{2} \right)^{2}{\psi_{11}}^{2}}{{\psi_{12}}^{2}}}
\end {array} \right]
\]
%\[
%Ric=-\frac 12 \left[ \begin {array}{cccccc}
%{\frac {\left( 1+{\psi_{11}}^{2}\right)^{4}+{\psi_{45}}^{2}{\psi_{12}}^{4}}{{\psi_{12}}^{4}}}&{\frac { \left( 1+{\psi_{11}}^{2} \right) ^{3}\psi_{11}}{{\psi_{12}}^{3}}}&0&0&0&0
%\\
%\noalign{\medskip}{\frac { \left( 1+{\psi_{11}}^{2}
% \right)^{3}\psi_{11}}{{\psi_{12}}^{3}}}&{\frac {\left(1+
%{\psi_{11}}^{2} \right)^{2}{\psi_{11}}^{2}}{{\psi_{12}}^{2}}}&0
%&0&0&0\\
%\noalign{\medskip}0&0&0&0&0&0\\
%\noalign{\medskip}0&0&0&0&0&0
%\\ \noalign{\medskip}0&0&0&0&0&0\\
%\noalign{\medskip}0&0&0&0&0&0
%\end {array} \right]
%\]
The Ricci tensor is not $J$-Hermitian for any values of the parameters.

Let us assume that all remaining parameters on which the Ricci tensor does not depend are equal to zero. Then we obtain the following expressions for the canonical almost complex structure and associated almost pseudo-K\"{a}hler metric:
\[
J(e_2)=\psi_{12}\, e_1 -\psi_{11}\, e_2,\qquad
J(e_3)=-\psi_{11}\, e_3 -\psi_{12}\, e_6,\qquad
J(e_5)=\psi_{45}\, e_4.
\]
%$$
%J= \left[ \begin {array}{cccccc} {\it psi11}&{\it psi12}&0&0&0&0
%\\ \noalign{\medskip}-{\frac {1+{{\it psi11}}^{2}}{{\it psi12}}}&-{
%\it psi11}&0&0&0&0\\ \noalign{\medskip}0&0&-{\it psi11}&0&0&{\frac {1+
%{{\it psi11}}^{2}}{{\it psi12}}}\\ \noalign{\medskip}0&0&0&0&{\it
%psi45}&0\\ \noalign{\medskip}0&0&0&-{{\it psi45}}^{-1}&0&0
%\\ \noalign{\medskip}0&0&-{\it psi12}&0&0&{\it psi11}\end {array}
% \right]
%$$
\[
g=\left[ \begin {array}{cccccc}
0&0&-\psi_{11}&0&0&{\frac {1+{\psi_{11}}^{2}}{\psi_{12}}}\\
\noalign{\medskip}0&0&-\psi_{12}&0&0&\psi_{11}\\
\noalign{\medskip}-\psi_{11}&-\psi_{12}&0&0&0&0\\
\noalign{\medskip}0&0&0&{\psi_{45}}^{-1}&0&0\\
\noalign{\medskip}0&0&0&0&\psi_{45}&0\\
\noalign{\medskip}{\frac {1+{\psi_{11}}^{2}}{\psi_{12}}}&\psi_{11}&0&0&0&0
\end {array} \right]
\]

\textbf{Second case.} The symplectic structure is:
$$
\omega_2 = e^1\wedge e^6 +e^2\wedge e^5 -e^3\wedge e^4.
$$

We require that the almost complex structure $J$ is $\omega_2$-compatible and that the ideal $\g_3$ is $J$-invariant. In this case the ideal $C^2\g$ is also $J$-invariant. The almost complex structure $J$ is almost nilpotent. The associated metric will be pseudo-Riemannian. We suppose also that $\psi_{61}=0$.

The condition $J^2=-1$ gives a complicated expression for the almost complex structure that depends on seven parameters (under the condition that $\psi_{61} =0$). Nevertheless, the curvature tensor is easily calculated, but has many non-zero and complicated components. The scalar square of the curvature tensor $R$ is equal to zero, $g(R,R)=0$. The scalar curvature $S$ is equal to zero.

The Ricci tensor has a block appearance with one non-zero left upper block of an order 2:
\[
Ric_{2\times 2} =-\frac 12\, \left[ \begin {array}{cc}
{{\psi_{11}}}^{2}{{\psi_{12}}}^{2}& {\psi_{11}}\,{{\psi_{12}}}^{3}\\
\noalign{\medskip}{\psi_{11}}\,{{\psi_{12}}}^{3}& {{\psi_{12}}}^{4}
\end {array} \right].
\]
%$$
%Ric=-\frac 12\, \left[ \begin {array}{cccccc} {{\psi_{11}}}^{2}{{\psi_{12}}}^{2
%}& {\psi_{11}}\,{{\psi_{12}}}^{3}&0&0&0&0\\
%\noalign{\medskip}{\psi_{11}}\,{{\psi_{12}}}^{3}& {{\psi_{12}}}^{4}&0&0&0&0
%\\ \noalign{\medskip}0&0&0&0&0&0\\ \noalign{\medskip}0&0&0&0&0&0
%\\ \noalign{\medskip}0&0&0&0&0&0\\ \noalign{\medskip}0&0&0&0&0&0
%\end {array} \right].
%$$
It depends on two parameters $\psi_{11}$ and $\psi_{12}\ne 0$ and is not $J$-Hermitian.

Let us assume that all remaining parameters on which the Ricci tensor does not depend are equal to zero. As $\psi_{34}\ne 0$, we consider $\psi_{34}=1$.
Then we obtain the following expressions for the canonical almost complex structure and associated almost pseudo-K\"{a}hler metric (at $\psi_{34}=1$):
\[
J(e_2)= \psi_{12}\, e_1 -\psi_{11}\, e_2,\qquad
J(e_4)=  e_3,\qquad
J(e_6)= -\psi_{12}\, e_5 -\psi_{11}\, e_6.
\]
%$$
%J = \left[ \begin {array}{cccccc} {\psi_{11}}&{\psi_{12}}&0&0&0&0
%\\ \noalign{\medskip}-{\frac {1+{{\psi_{11}}}^{2}}{{\psi_{12}}}}&-{
%\psi_{11}}&0&0&0&0\\ \noalign{\medskip}0&0&0&1&0&0
%\\ \noalign{\medskip}0&0&-1&0&0&0\\ \noalign{\medskip}0&0&0&0&{\psi_{11}}&-{\psi_{12}}\\ \noalign{\medskip}0&0&0&0&{\frac {1+{{\psi_{11}}
%}^{2}}{{\psi_{12}}}}&-{\psi_{11}}\end {array} \right],
%$$
$$
g= \left[ \begin {array}{cccccc} 0&0&0&0&{\frac {1+{{\psi_{11}}}^{2}}{{
\psi_{12}}}}&-{\psi_{11}}\\ \noalign{\medskip}0&0&0&0&{\psi_{11}}&-{
\psi_{12}}\\ \noalign{\medskip}0&0&1&0&0&0\\ \noalign{\medskip}0&0&0&1
&0&0\\ \noalign{\medskip}{\frac {1+{{\psi_{11}}}^{2}}{{\psi_{12}}}}&{
\psi_{11}}&0&0&0&0\\ \noalign{\medskip}-{\psi_{11}}&-{\psi_{12}}&0&0&0
&0\end {array} \right].
$$

\subsection{The Lie group $G_{8}$}
Let us consider a six-dimensional symplectic Lie group $G_8$ which has a Lie algebra $\g_8$  with non-trivial Lie brackets \cite{Goze-Khakim-Med}:

$[e_1,e_3] = e_4$, $[e_1,e_4] = e_5$,  $[e_1,e_5] = e_6$, $[e_2,e_3] = e_5$ , $[e_2,e_4] = e_6$.\\
This Lie algebra is semidirect product of a five-dimensional Lie algebra $\mathfrak{h}= \R\{e_1,e_3,e_4,e_5,e_6\}$ and $\mathbb{R}^1=\mathbb{R}\{e_2\}$.
%$\g_{8}=(0,0,0,13,14+23,15+24)$.\\
The sequence of ideals is:
$C^1\g=\R\{e_4,e_5,e_6\}$,
$C^2\g=\R\{e_5,e_6\}$,
$C^3\g=\R\{e_6\}=\mathcal{Z}$, and the last is the Lie algebra center.
There is an increasing series of ideals: $\g_1=\mathcal{Z}=\R\{e_6\}$,\, $\g_2=\R\{e_2,e_5,e_6\}$,\,
$\g_3=\R\{e_2,e_4,e_5,e_6\}$,\,
$\g_4=\g$. The Lie algebra type is (1,3,4,6).

The symplectic structure is defined by the 2-form \cite{Goze-Khakim-Med}:
$$
\omega = e^1\wedge e^6 +e^2\wedge e^5 -e^3\wedge e^4.
$$

We require that the almost complex structure $J$ is $\omega$-compatible and that the ideal $\g_3$ is $J$-invariant. In this case the ideal $C^2\g$ is also $J$-invariant. The almost complex structure $J$ is almost nilpotent. The associated metric will be pseudo-Riemannian. We suppose also that $\psi_{61}=0$.

The condition $J^2=-1$ gives a complicated expression for the almost complex structure that depends on eight parameters. Nevertheless, the curvature tensor is easily calculated, but has many non-zero and complicated components. The scalar square of the curvature tensor $R$ is equal to zero, $g(R,R)=0$. The scalar curvature $S$ is equal to zero.

Then Ricci tensor has a block appearance with one non-zero left upper block of an order 2 and depends on three parameters $\psi_{11}$, $\psi_{12}\ne 0$ and $\psi_{34}\ne 0$.
It is $J$-Hermitian at $\psi_{34}=\pm\psi_{12}$. Then:
\[
Ric_{2\times 2} =-\frac 12\left[ \begin {array}{cc} {{\psi_{12}}}^{2}(1+{{\psi_{11}}}^{2})& {\psi_{11}}\,{{\psi_{12}}}^{3} \\ \noalign{\medskip}{\psi_{11}}\,{{\psi_{12}}}^{3}& {{\psi_{12}}}^{4}
\end {array} \right].
\]
%$$
%Ric= -\frac 12\left[ \begin {array}{cccccc} {{\psi_{11}}}^{2}{{\psi_{12}}}^{2}+{{\psi_{12}}}^{2}& {\psi_{11}}\,{{\psi_{12}}}^{3}&0&0&0&0 \\ \noalign{\medskip}{\psi_{11}}\,{{\psi_{12}}}^{3}& {{\psi_{12}}}^{4}&0&0&0&0\\ \noalign{\medskip}0&0&0&0&0&0 \\
%\noalign{\medskip}0&0&0&0&0&0\\
%\noalign{\medskip}0&0&0&0&0&0 \\
%\noalign{\medskip}0&0&0&0&0&0\end {array} \right].
%$$

Let us assume that all remaining parameters on which the Ricci tensor does not depend are equal to zero. Then we obtain the following expressions for the canonical almost complex structure $J$ and associated almost pseudo-K\"{a}hler metric $g$ with a Hermitian Ricci tensor (at $\psi_{34}=\psi_{12}$):
\[
J(e_2)= \psi_{12}\, e_1 -\psi_{11}\, e_2,\qquad
J(e_4)= \psi_{12}\, e_3,\qquad
J(e_6)= -\psi_{12}\, e_5 -\psi_{11}\, e_6.
\]
%$$
%J=  \left[ \begin {array}{cccccc} {\psi_{11}}&{\psi_{12}}&0&0&0&0
%\\ \noalign{\medskip}-{\frac {1+{{\psi_{11}}}^{2}}{{\psi_{12}}}}&-{
%\psi_{11}}&0&0&0&0\\ \noalign{\medskip}0&0&0&{\psi_{12}}&0&0
%\\ \noalign{\medskip}0&0&-{{\psi_{12}}}^{-1}&0&0&0
%\\ \noalign{\medskip}0&0&0&0&{\psi_{11}}&-{\psi_{12}}
%\\ \noalign{\medskip}0&0&0&0&{\frac {1+{{\psi_{11}}}^{2}}{{\psi_{12}}}
%}&-{\psi_{11}}\end {array} \right],
%$$
$$
g=  \left[ \begin {array}{cccccc} 0&0&0&0&{\frac {1+{{\psi_{11}}}^{2}}{{
\psi_{12}}}}&-{\psi_{11}}\\ \noalign{\medskip}0&0&0&0&{\psi_{11}}&-{
\psi_{12}}\\ \noalign{\medskip}0&0&{{\psi_{12}}}^{-1}&0&0&0
\\ \noalign{\medskip}0&0&0&{\psi_{12}}&0&0\\ \noalign{\medskip}{\frac
{1+{{\psi_{11}}}^{2}}{{\psi_{12}}}}&{\psi_{11}}&0&0&0&0
\\ \noalign{\medskip}-{\psi_{11}}&-{\psi_{12}}&0&0&0&0\end {array}
 \right].
$$

\subsection{The Lie group $G_{9}$}
Let us consider a six-dimensional symplectic Lie group $G_9$ which has a Lie algebra $\g_9$  with non-trivial Lie brackets \cite{Goze-Khakim-Med}:

$[e_1,e_2] = e_4$, $[e_1,e_4] = e_5$, $[e_1,e_5] = e_6$, $[e_2,e_3] = e_6$.\\
This Lie algebra is semidirect product of a five-dimensional Lie algebra $\mathfrak{h}= \R\{e_1,e_2,e_4,e_5,e_6\}$ and $\mathbb{R}^1=\mathbb{R}\{e_3\}$.
%$\g_{9}=(0,0,0,12,14,15+23)$.\\
The sequence of ideals is:
$C^1\g=\R\{e_4,e_5,e_6\}$,
$C^2\g=R\{e_5,e_6\}$
$\mathcal{Z}= \R\{e_6\}$, and the last is the Lie algebra center.
There is an increasing series of ideals: $\g_1=\mathcal{Z}=\R\{e_6\}$,\, $\g_2=\R\{e_3,e_5,e_6\}$,\,
$\g_3=\R\{e_3,e_4,e_5,e_6\}$,\,
$\g_4=\g$. The Lie algebra type is (1,3,4,6).

The symplectic structure is defined by the 2-form \cite{Goze-Khakim-Med}:
$$
\omega = e^1\wedge e^3 +e^2\wedge e^6 -e^4\wedge e^5.
$$

We require that the almost complex structure $J$ is $\omega$-compatible and that the ideal $\g_3$ is $J$-invariant. In this case the ideal $I=\R\{e_3,e_6\}$ is also $J$-invariant. The almost complex structure $J$ is almost nilpotent. The associated metric will be pseudo-Riemannian. We suppose also that $\psi_{61}=0$.

The condition $J^2=-1$ gives a complicated expression for the almost complex structure that depends on seven parameters (at $\psi_{61}=0$). Nevertheless, the curvature tensor is easily calculated, but has many non-zero and complicated components. The scalar square of the curvature tensor $R$ is equal to zero, $g(R,R)=0$. The scalar curvature $S$ is equal to zero.

Then Ricci tensor has a block appearance with one non-zero left upper block of an order 2:
\[
Ric_{2\times 2}=-\frac 12 \left[ \begin {array}{cc}
{\frac {\left( 1+{\psi_{11}}^{2}\right)^{4}+{\psi_{45}}^{2}{\psi_{12}}^{4}}{{\psi_{12}}^{4}}}&{\frac { \left( 1+{\psi_{11}}^{2} \right) ^{3}\psi_{11}}{{\psi_{12}}^{3}}}\\
\noalign{\medskip}{\frac { \left( 1+{\psi_{11}}^{2}
 \right)^{3}\psi_{11}}{{\psi_{12}}^{3}}}&{\frac {\left(1+
{\psi_{11}}^{2} \right)^{2}{\psi_{11}}^{2}}{{\psi_{12}}^{2}}}
\end {array} \right].
\]
The Ricci tensor is not $J$-Hermitian for any values of the parameters.

Let us assume that all remaining parameters on which the Ricci tensor does not depend are equal to zero. Then we obtain the following expressions for the canonical almost complex structure and associated almost pseudo-K\"{a}hler metric:
\[
J(e_2)=\psi_{12}\, e_1 -\psi_{11}\, e_2,\qquad
J(e_3)=-\psi_{11}\, e_3 -\psi_{12}\, e_6,\qquad
J(e_5)=\psi_{45}\, e_4.
\]
%$$
%J= \left[ \begin {array}{cccccc} {\it psi11}&{\it psi12}&0&0&0&0\\ \noalign{\medskip}-{\frac {1+{{\it psi11}}^{2}}{{\it psi12}}}&-{
%\it psi11}&0&0&0&0\\ \noalign{\medskip}0&0&-{\it psi11}&0&0&{\frac {1+
%{{\it psi11}}^{2}}{{\it psi12}}}\\ \noalign{\medskip}0&0&0&0&{\it
%psi45}&0\\ \noalign{\medskip}0&0&0&-{{\it psi45}}^{-1}&0&0
%\\ \noalign{\medskip}0&0&-{\it psi12}&0&0&{\it psi11}\end {array}
% \right],
%$$
\[
g=\left[ \begin {array}{cccccc}
0&0&-\psi_{11}&0&0&{\frac {1+{\psi_{11}}^{2}}{\psi_{12}}}\\
\noalign{\medskip}0&0&-\psi_{12}&0&0&\psi_{11}\\
\noalign{\medskip}-\psi_{11}&-\psi_{12}&0&0&0&0\\
\noalign{\medskip}0&0&0&{\psi_{45}}^{-1}&0&0\\
\noalign{\medskip}0&0&0&0&\psi_{45}&0\\
\noalign{\medskip}{\frac {1+{\psi_{11}}^{2}}{\psi_{12}}}&\psi_{11}&0&0&0&0
\end {array} \right]
\]

\subsection{The Lie group $G_{19}$}
Let us consider a six-dimensional symplectic Lie group $G_{19}$ which has a Lie algebra $\g_{19}$  with non-trivial Lie brackets \cite{Goze-Khakim-Med}:

$[e_1,e_2] = e_4$, $[e_1,e_4] = e_5$, $[e_1,e_5] = e_6$.\\
This Lie algebra is direct product of a five-dimensional Lie algebra $\mathfrak{h}= \R\{e_1,e_2,e_4,e_5,e_6\}$ and $\mathbb{R}^1=\mathbb{R}\{e_3\}$.
%$\g_{19}=(0,0,0,12,14,15)$.\\
The sequence of ideals is:
$D^1\g=C^1\g=\R\{e_4,e_5,e_6\}$, $D^2\g=0$,
$C^2\g=\R\{e_5,e_6\}$, $C^3\g=\R\{e_6\}$,
$\mathcal{Z}= \R\{e_3,e_6\}$, and the last is the Lie algebra center.
There is an increasing series of ideals: $\g_1=\mathcal{Z}=\R\{e_3,e_6\}$,\, $\g_2=\R\{e_3,e_5,e_6\}$,\,
$\g_3=\R\{e_3,e_4,e_5,e_6\}$,\,
$\g_4=\g$. The Lie algebra type is (2,3,4,6).

The symplectic structure is defined by the 2-form \cite{Goze-Khakim-Med}:
$$
\omega = e^1\wedge e^3 +e^2\wedge e^6 -e^4\wedge e^5.
$$

We require that the almost complex structure $J$ is $\omega$-compatible and that the ideal $\g_3$ is $J$-invariant. In this case the ideal $I=\R\{e_3,e_6\}$ is also $J$-invariant. The almost complex structure $J$ is almost nilpotent. The associated metric will be pseudo-Riemannian. We suppose also that $\psi_{62}=0$.

The condition $J^2=-1$ gives a complicated expression for the almost complex structure that depends on seven parameters (at $\psi_{61}=0$). Nevertheless, the curvature tensor is easily calculated, but has many non-zero and complicated components. The scalar square of the curvature tensor $R$ is equal to zero, $g(R,R)=0$. The scalar curvature $S$ is equal to zero.

The Ricci tensor has only one component:
\[
Ric= -\frac 12\, (\psi_{45})^2\, (e_1)^2.
\]
The Ricci tensor is not $J$-Hermitian for any values of the parameters.

Let us assume that all remaining parameters on which the Ricci tensor does not depend are equal to zero. As $\psi_{12}\ne 0$, we consider $\psi_{12}=1$. Then we obtain the following expressions for the canonical almost complex structure $J$ and associated almost pseudo-K\"{a}hler metric $g$:
\[
J(e_2)= e_1,\qquad
J(e_5)=\psi_{45}\, e_4,\qquad
J(e_6)= e_3.
\]
%$$
%J=\left[ \begin {array}{cccccc} 0&1&0&0&0&0\\ \noalign{\medskip}-1&0&0&0&0&0\\ \noalign{\medskip}0&0&0&0&0&1\\ \noalign{\medskip}0&0&0&0&{\it
%psi45}&0\\ \noalign{\medskip}0&0&0&-{{\it psi45}}^{-1}&0&0
%\\ \noalign{\medskip}0&0&-1&0&0&0\end {array} \right],
%$$
$$
g= \left[ \begin {array}{cccccc}
0&0&0&0&0&1\\
\noalign{\medskip}0&0&-1&0&0&0\\
\noalign{\medskip}0&-1&0&0&0&0\\
\noalign{\medskip}0&0&0&{\psi_{45}}^{-1}&0&0\\
\noalign{\medskip}0&0&0&0&\psi_{45}&0\\
\noalign{\medskip}1&0&0&0&0&0
\end {array} \right].
$$

\subsection{The Lie group $G_{20}$}
Let us consider a six-dimensional symplectic Lie group $G_{20}$ which has a Lie algebra $\g_{20}$  with non-trivial Lie brackets \cite{Goze-Khakim-Med}:

$[e_1,e_2] = e_3$, $[e_1,e_3] = e_4$, $[e_1,e_4] = e_5$, $[e_2,e_3] = e_5$.\\
This Lie algebra is direct product of a five-dimensional Lie algebra $\mathfrak{h}= \R\{e_1,e_2,e_3,e_4,e_5\}$ and $\mathbb{R}^1=\mathbb{R}\{e_6\}$
%$\g_{20}=(0,0,12,13,14+23,0)$.\\

The sequence of ideals is:
$D^1\g=C^1\g=\R\{e_3,e_4,e_5\}$, $D^2\g=0$,
$C^2\g=\R\{e_4,e_5\}$,
$C^3\g=\R\{e_5\}$,
$\mathcal{Z}= \R\{e_5,e_6\}$, and the last is the Lie algebra center.
There is an increasing series of ideals: $\g_1=\mathcal{Z}=\R\{e_5,e_6\}$,\, $\g_2=\R\{e_4,e_5,e_6\}$,\,
$\g_3=\R\{e_3,e_4,e_5,e_6\}$,\,
$\g_4=\g$. The Lie algebra type is (2,3,4,6).

The symplectic structure is defined by the 2-form \cite{Goze-Khakim-Med}:
$$
\omega = e^1\wedge e^6 +e^2\wedge e^5  -e^3\wedge e^4.
$$

We require that the almost complex structure $J$ is $\omega$-compatible and that the center $\mathcal{Z}$ is $J$-invariant.
In this case the ideal $\g_3$ is also $J$-invariant. The almost complex structure $J$ is almost nilpotent. The associated metric will be pseudo-Riemannian.

As $\mathcal{Z}=\R\{e_5,e_6\}$, then according to Lemma \ref{ParRiem} and Corollary \ref{ParCurvTensor}, it is possible to assume that some of the free parameters ($\psi_{ij}$, $i=5,6,\, j=1,\dots, 6$) are equal to zero (the curvature does not depend on them). Such free parameters will be $\psi_{51}$, $\psi_{52}$ and $\psi_{61}$. We assume that they are zero.

The condition $J^2=-1$ gives a complicated expression for the almost complex structure that depends on seven parameters (at $\psi_{61}=0$). Nevertheless, the curvature tensor is easily calculated, but has many non-zero and complicated components. The scalar square of the curvature tensor $R$ is equal to zero, $g(R,R)=0$. The scalar curvature $S$ is equal to zero.

The Ricci tensor has only one component and is not $J$-Hermitian:
\[
Ric= -\frac 12{\psi_{34}}^{2}\, (e^1)^2.
\]
%$$
%Ric= -\frac 12\left[ \begin {array}{cccccc} {{\psi_{34}}}^{2}&0&0&0&0&0
%\\ \noalign{\medskip}0&0&0&0&0&0\\ \noalign{\medskip}0&0&0&0&0&0
%\\ \noalign{\medskip}0&0&0&0&0&0\\ \noalign{\medskip}0&0&0&0&0&0
%\\ \noalign{\medskip}0&0&0&0&0&0\end {array} \right].
%$$

Let us assume that all remaining parameters on which the Ricci tensor does not depend are equal to zero. As $\psi_{12}\ne 0$, we consider $\psi_{12}=1$. Then we obtain the following expressions for the canonical almost complex structure $J$ and associated almost pseudo-K\"{a}hler metric $g$:
$$
J_0(e_2)=e_1,\quad J_0(e_4)=\psi_{34}e_3,\quad J_0(e_5)=e_6,
$$
$$
g_0=2\, e^1\cdot e^5 -2\, e^2\cdot e^6 +\frac {1}{\psi_{34}}(e^3)^2 + \psi_{34}(e^4)^2.
$$
%
%$$
%J_0= \left[ \begin {array}{cccccc} 0&1&0&0&0&0\\ \noalign{\medskip}-1&0&0&0
%&0&0\\ \noalign{\medskip}0&0&0&{\psi_{34}}&0&0\\ \noalign{\medskip}0&0
%&-{{\psi_{34}}}^{-1}&0&0&0\\ \noalign{\medskip}0&0&0&0&0&-1
%\\ \noalign{\medskip}0&0&0&0&1&0\end {array} \right],
%$$
%
%$$
%g_0= \left[ \begin {array}{cccccc} 0&0&0&0&1&0\\ \noalign{\medskip}0&0&0&0
%&0&-1\\ \noalign{\medskip}0&0&{{\psi_{34}}}^{-1}&0&0&0
%\\ \noalign{\medskip}0&0&0&{\psi_{34}}&0&0\\ \noalign{\medskip}1&0&0&0
%&0&0\\ \noalign{\medskip}0&-1&0&0&0&0\end {array} \right].
%$$

\subsection{The Lie group $G_{22}$}
Let us consider a six-dimensional symplectic Lie group $G_{22}$ which has a Lie algebra $\g_{22}$  with non-trivial Lie brackets \cite{Goze-Khakim-Med}:

$[e_1,e_2] = e_5$, $[e_1,e_5] = e_6$. \\
%$\g_{22}=(0,0,0,0,12,15)$.\\
This Lie algebra is direct product of a four-dimensional Lie algebra $\mathfrak{h}= \R\{e_1,e_2,e_5,e_6\}$ and $\mathbb{R}^2 = \R\{e_3,e_4\}$.

The sequence of ideals is:
$C^1\g=\R\{e_5,e_6\}$, $C^2\g=\R\{e_6\}$,
$\mathcal{Z}= \R\{e_3,e_4,e_6\}$, and the last is the Lie algebra center.
There is an increasing series of ideals: $\g_1=\mathcal{Z}=\R\{e_3,e_4,e_6\}$,\, $\g_2=\R\{e_3,e_4,e_5,e_6\}$,\,
$\g_3=\g$. The Lie algebra type is (3,4,6).

The symplectic structure is defined by the 2-form \cite{Goze-Khakim-Med}:
$$
\omega = e^1\wedge e^6 +e^2\wedge e^5  +e^3\wedge e^4.
$$

We require that the almost complex structure $J$ is $\omega$-compatible and that the ideal $C^1\g$ is $J$-invariant.
In this case the ideal $\g_2$ is also $J$-invariant. The almost complex structure $J$ is almost nilpotent. The associated metric will be pseudo-Riemannian.

The condition $J^2=-1$ gives a complicated expression for the almost complex structure that depends on seven parameters (at $\psi_{61}=0$). Nevertheless, the curvature tensor is easily calculated, but has many non-zero and complicated components. The scalar square of the curvature tensor $R$ is equal to zero, $g(R,R)=0$. The scalar curvature $S$ is equal to zero.

The Ricci tensor has a block appearance with one non-zero left upper block of an order 2:
\[
Ric_{2\times 2} =-\frac 12\left[ \begin {array}{cc}
{{\psi_{11}}}^{2}{{\psi_{12}}}^{2}& {\psi_{11}}\,{{\psi_{12}}}^{3}\\ \noalign{\medskip}{\psi_{11}}\,{{\psi_{12}}}^{3}& {{\psi_{12}}}^{4}
\end {array} \right].
\]
%$$
%Ric=-\frac 12\left[ \begin {array}{cccccc} {{\psi_{11}}}^{2}{{\psi_{12}}}^{2}& {\psi_{11}}\,{{\psi_{12}}}^{3}&0&0&0&0\\ \noalign{\medskip}{\psi_{11}}\,{{\psi_{12}}}^{3}& {{\psi_{12}}}^{4}&0&0&0&0
%\\ \noalign{\medskip}0&0&0&0&0&0\\ \noalign{\medskip}0&0&0&0&0&0
%\\ \noalign{\medskip}0&0&0&0&0&0\\ \noalign{\medskip}0&0&0&0&0&0
%\end {array} \right].
%$$
It depends on two parameters $\psi_{11}$ and $\psi_{12}\ne 0$ and is not $J$-Hermitian.

Let us assume that all remaining parameters on which the Ricci tensor does not depend are equal to zero. As $\psi_{34}\ne 0$, we consider $\psi_{34}=-1$. Then we obtain the following expressions for the canonical almost complex structure and associated almost pseudo-K\"{a}hler metric:
\[
J(e_2)= \psi_{12}\, e_1 -\psi_{11}\, e_2,\qquad
J(e_4)= -e_3,\qquad
J(e_6)= -\psi_{12}\, e_5 -\psi_{11}\, e_6.
\]
%$$
%J=  \left[ \begin {array}{cccccc} {\psi_{11}}&{\psi_{12}}&0&0&0&0
%\\ \noalign{\medskip}-{\frac {1+{{\psi_{11}}}^{2}}{{\psi_{12}}}}&-{
%\psi_{11}}&0&0&0&0\\ \noalign{\medskip}0&0&0&-1&0&0
%\\ \noalign{\medskip}0&0&1&0&0&0\\ \noalign{\medskip}0&0&0&0&{\psi_{11}}&-{\psi_{12}}\\ \noalign{\medskip}0&0&0&0&{\frac {1+{{\psi_{11}}
%}^{2}}{{\psi_{12}}}}&-{\psi_{11}}\end {array} \right],
%$$
$$
g= \left[ \begin {array}{cccccc} 0&0&0&0&{\frac {1+{{\psi_{11}}}^{2}}{{
\psi_{12}}}}&-{\psi_{11}}\\ \noalign{\medskip}0&0&0&0&{\psi_{11}}&-{
\psi_{12}}\\ \noalign{\medskip}0&0&1&0&0&0\\ \noalign{\medskip}0&0&0&1
&0&0\\ \noalign{\medskip}{\frac {1+{{\psi_{11}}}^{2}}{{\psi_{12}}}}&{
\psi_{11}}&0&0&0&0\\ \noalign{\medskip}-{\psi_{11}}&-{\psi_{12}}&0&0&0
&0\end {array} \right].
$$

\section{Formulas for evaluations}
We now present the formulas which were used for the evaluations on Maple of curvature tensor of the associated metrics. Let $e_1,\ldots,e_{2n}$ be a basis of the Lie algebra $\mathfrak g$ and $C_{ij}^k$ a structure constant of the Lie algebra in this base:
\begin{equation}
[e_i,e_j]=\sum_{k=1}^{2n}C_{ij}^{k}e_k,  \label{strukt}
\end{equation}

\textbf{1. Compatible condition.}
This is the condition that $\omega(JX,Y) + \omega(X, JY) =0$, $\forall \, X, Y\in \g $. For the basis vectors we have:
$\omega (J (e_i), e_j) + \omega (e_i, J (e_j)) =0$, \quad
$\omega (J^k_i e_k, e_j) + \omega (e_i, J^s_j e_s) =0$.
\begin{equation}
\omega _ {k, j} J^k_i + \omega _ {i, s} J^s_j=0.
\end{equation}

\textbf{2. Connection components.} These are the components $\Gamma_{ij}^{k}$ in the formula
$\nabla _ {e _ {i}} e_j =\Gamma_{ij}^{k}e_k.$
For left-invariant vector fields we have: $2g ({\nabla}_{X} Y, Z) =g ([X, Y], Z) +g ([Z, X], Y)-g ([Y, Z], X)$. For the basis vectors we have:
$$
2g (\nabla _ {e _ {i}} e_j, e_k) =g ([e_i, e_j], e_k) +g ([e_k, e_i], e_j) +g (e_i, [e_k, e_j]),
$$
$$
2g _ {lk} \Gamma _ {ij} ^ {l} =g _ {pk} C _ {ij} ^ {p} +g _ {pj} C _ {ki} ^ {p} +g _ {ip} C _ {kj} ^ {p},
$$
\begin{equation}
\Gamma_{ij}^{n}=\frac{1}{2}g^{kn}\left(g_{pk}C_{ij}^{p}+g_{pj}C_{ki}^{p} +g_{ip}C_{kj}^{p}\right).
\end{equation}

\textbf{3. Curvature tensor.}
The formula is: $R(X, Y)Z =\nabla_{X} \nabla_{Y} Z -\nabla_{Y}\nabla_{X}Z -\nabla_{[X,Y]}Z$.
For the basis vectors we have: $R (e_i, e_j) e_k=R _ {ijk} ^s e_s $,
$$
R(e_i,e_j)e_k=\nabla_{e_{i}}\nabla_{e_{j}}e_{k}-\nabla_{e_{j}}\nabla_{e_{i}}e_{k} -\nabla_{[e_{i},e_{j}]}e_{k}.
$$
Therefore:
\begin{equation}
R_{ijk}^{s}=\Gamma_{ip}^{s}\Gamma_{jk}^{p}-\Gamma_{jp}^{s}\Gamma_{ik}^{p} -C_{ij}^{p}\Gamma_{pk}^{s}.
\end{equation}

%\newpage
\

\end{document}